\documentclass[11pt]{amsart}

\usepackage{amsmath,amssymb,amsthm,amscd,latexsym}
\usepackage[frame,cmtip,arrow,matrix,line,graph,curve]{xy}
\usepackage{graphpap, color}
\usepackage[mathscr]{eucal}
\usepackage{color}

\usepackage{color}
\newtheorem{theorem}{Theorem}

\newtheorem{lemma}[theorem]{Lemma}
\newtheorem{corollary}[theorem]{Corollary}
\newtheorem{proposition}[theorem]{Proposition}

\newtheorem{example}[theorem]{Example}

\numberwithin{theorem}{section}
\numberwithin{equation}{section}

\title{ A class of curvature type equations}

\dedicatory{Dedicated to Duong H. Phong on the occasion of his 65th birthday}
\author{Pengfei Guan}
\address{Department of Mathematics and Statistics\\ McGill University \\ Montreal, Canada}
\email{guan@math.mcgill.ca}

\author{Xiangwen Zhang}
\address{Department of Mathematics\\ University of California, Irvine \\ California, US}
\email{xiangwen@math.uci.edu}

\date{\today}
\thanks{Research of the first author was supported in part by an NSERC Discovery Grant. Research of the second author was supported by the NSF under Grant DMS-1809582.}

\begin{document}

\begin{abstract}
 In this paper, we study the solvability of a general class of fully nonlinear curvature equations, which can be viewed as generalizations of the equations for Christoffel-Minkowski problem in convex geometry. We will also study the Dirichlet problem of the corresponding degenerate equations as an extension of the equations studied by Krylov.
 \end{abstract}

\maketitle


\section{Introduction}
\par
In this paper, we consider the solvability of the following fully nonlinear elliptic equation on sphere $\mathbb S^n$,
\begin{eqnarray}\label{equ-s}
\sigma_k(W_u(x))+ \alpha(x)\sigma_{k-1}(W_u(x)) =\sum_{\ell=0} ^{k-2} \alpha_{\ell}(x) \, \sigma_{\ell}(W_u(x)), \  \ x\in \mathbb S^n
\end{eqnarray}
with $W_u(x) = u_{ij}(x) + u(x)\, \delta_{ij}$ being the spherical hessian matrix of the unknown function $u:\, \mathbb S^n \rightarrow \mathbb R$. Here $u_{ij}$ are the second order covariant derivatives with respect to any orthonormal frame $\{e_1, e_2, \cdots, e_n\}$ on $\mathbb S^n$ and $\delta_{ij}$ is the standard Kronecker symbol. $\alpha(x)$ and $\alpha_{\ell}(x)>0$ with $\ell= 0, 1\cdots, k-2$ are given smooth functions on $\mathbb S^n$. $\sigma_m(A)$ denotes the $m$-th elementary symmetric function of an $n\times n$ symmetric matrix $A$ given by
\[
\sigma_m(A) = \sigma_m(\lambda(A)) = \sum_{i_1 < i_2< \cdots <i_m} \lambda_{i_1} \, \lambda_{i_2} \, \cdots \lambda_{i_m}
\]
where $\lambda(A)= (\lambda_1, \lambda_2, \cdots, \lambda_n)$ denotes the eigenvalues of the matrix $A$ for $1\leq m \leq n$ and $\sigma_0(A) =1$.

\

\par
Our interest on the solvability of equation (\ref{equ-s}) is motivated from the study of convex geometry. The classical Christoffel-Minkowski problem is a problem of finding a convex hypersurface in $\mathbb R^{n+1}$ with the $k$-th symmetric function of the principal radii prescribed on its outer normals. It corresponds to finding convex solutions of the following standard Hessian equations
\begin{eqnarray*}
 \sigma_k (W_u(x) ) = \phi(x), \ \ x\in \mathbb S^n
\end{eqnarray*}
with the positive definite condition $W_u>0$ on $\mathbb S^n$. In the case $k=1$, it is the equation for the Christoffel problem and it was solved by Firey \cite{Firey1, Firey2} and Berg \cite{Berg}. When $k=n$, this is the famous Minkowski problem and it has been settled by the works of Minkowski \cite{Mink}, Alexandrov \cite{Alex}, Lewy \cite{Lewy}, Nirenberg \cite{Nirenberg0}, Pogorelov \cite{Pog} and Cheng-Yau \cite{Cheng-Yau}. The intermediate case was solved by Guan-Ma \cite{Guan-Ma}. It is obvious that the Hessian equation corresponds to the case that $\alpha(x) \equiv 0$ and $\alpha_\ell (x) \equiv 0$ for $1\leq \ell \leq k-2$ in equation (\ref{equ-s}).

\medskip

\noindent
Another classical problem from convex geometry is to find convex hypersurfaces in $\mathbb R^{n+1}$ whose Weigarten curvatures is prescribed as a function defined on $\mathbb S^n$ in terms of the inverse Gauss map. It corresponds to find convex solutions of the quotient type equation
\[
{\sigma_n\over \sigma_{n-k}}(W_u(x)) = \phi(x), \ \ x \in \mathbb S^n,
\]
where $\phi(x)$ is a given positive smooth function on $\mathbb S^n$. This problem has been extensively studied and important progress has been made in Guan-Guan \cite{Guan-Guan} and Guan-Li-Li \cite{Guan-Li-Li}. One finds that the above equation is also contained in the frame of equation (\ref{equ-s}) if we move $\sigma_{n-k}(W_u)$ to the right hand side of the equation.

\medskip

\par

Various extensions of the Christoffel-Minkowski problem have been discussed in the literature of convex geometry. For example, the problem of prescribing convex combination of area measures was proposed in \cite{Schneider}. This type of problem leads to differential equation of the form
\begin{eqnarray}
 \sigma_k (W_u(x) )+\sum_{i=1}^{k-1} \alpha_i\,  \sigma_i(W_u(x)) = \phi(x), \ \ x\in \mathbb S^n
\end{eqnarray}
where $\alpha_i$ with $i=1,\cdots, k-1$ are nonnegative constants. The existence and uniqueness of convex solutions to this equation is still open if $\sum_{i=1}^{k-1} \alpha_i>0$ in general (we refer the discussion of uniqueness in the notes of section 7.2 in \cite{Schneider}). It would be interesting to consider the existence of convex hypersurfaces with prescribing more general functions on curvatures or principal radii. For that, we would like to investigate the solvability of equation (\ref{equ-s}) on sphere. In the first part of the paper, we will prove the following existence result.

\begin{theorem}\label{existence proposition}
Assume that $\alpha_\ell(x) \in C^{m, 1}(\mathbb S^n)$, with $m\geq 1$ and $0\leq \ell \leq k-2$, are positive functions and $\alpha(x)\in C^{m, 1}(\mathbb S^n)$. Suppose there is an automorphic group $\mathcal G$ of $\mathbb S^n$ that has no fixed points. If $\alpha(x), \alpha_\ell(x)$ are invariant under $\mathcal G$, i.e., $\alpha\left(g(x)\right) = \alpha(x)$ and $\alpha_{\ell}\left(g(x)\right) = \alpha_\ell(x)$ for all $g\in \mathcal G$ and $x\in \mathbb S^n$, then there exists a $\mathcal G$-invariant admissible solution $u\in C^{m+2, \gamma}, \forall\,  0<\gamma<1$, such that $u$ satisfies equation \eqref{equ-s}.
\par
Moreover, there is a positive constant $C$ depending only on $\inf_{\mathbb S^n} \alpha_\ell(x)$, $||\alpha_\ell ||_{C^{m, 1}(\mathbb S^n)}$ and $||\alpha||_{C^{m, 1}(\mathbb S^n)}$ such that
\begin{eqnarray}\label{est-m2}
|| u ||_{C^{m+2, \gamma}(\mathbb S^n)} \leq C.
\end{eqnarray}
\end{theorem}

\medskip

We remark that the $\mathcal G$-invariant assumption in the above theorem is used to prove the existence of solution by degree theory, which is in the similar case as \cite{Guan-Guan}. As explained in \cite{Guan-Guan}, such an assumption could be dropped if necessary geometric obstructions to the existence of solutions were at hand as all necessary a priori estimates for equation (\ref{equ-s}) are established without such invariance requirement.

\smallskip

\par
On the other hand, the solution obtained in the above theorem is the {\it admissible solution} in the sense that $\lambda(W_u)\in \Gamma_{k-1}$ cone with
\[
\Gamma_{k-1} = \{\lambda\in \mathbb R^n \, |\, \sigma_1(\lambda)>0, \cdots, \sigma_{k-1}(\lambda)>0\}.
\]
 This means the solution is known to be in the $\Gamma_{k-1}$ cone, but not necessarily in the convex cone $\Gamma_n=\{\lambda\in \mathbb R^n \, |\, \sigma_1(\lambda)>0, \cdots, \sigma_{n}(\lambda)>0\}$. Therefore, we can not fully recover the existence of convex hypersurface by using this existence result. The main ingredient missing here is a so-called {\it constant rank theorem} for equation (\ref{equ-s}) on $\mathbb S^n$. Such a theorem was established for the Hessian equation by Guan-Ma \cite{Guan-Ma} and it was also the key to obtain the convex solution. In their proof, homogeneous property of the equation plays important role there. It would be an interesting question to prove the constant rank theorem for fully nonlinear equations on $\mathbb S^n$ which are in-homogeneous.

\

Because of its structure as a combination of elementary symmetric functions, equation (\ref{equ-s}) is also interesting from the PDE point of view. Such type of equations arise naturally from many important geometric problems. One example is the so-called {\it Fu-Yau equation} arising from the study of the Hull-Strominger system in theoretical physics, which is an equation that can be written as the linear combination of the first and second elementary symmetric functions,
\[
\sigma_1 (i\partial\bar\partial (e^u+ \alpha' e^{-u} )) + \alpha' \, \sigma_2(i\partial\bar\partial u) = \phi
 \]
 on n-dimensional compact K\"ahler manifolds. There have been a lot of works related to this equation recently, see for example, Fu-Yau \cite{FY1, FY2} and Phong-Picard-Zhang \cite{PPZ1, PPZ, PPZ2}. Another important example is the {\it special Lagrangian equations} introduced by Harvey and Lawson \cite{HL}, which can be written as the alternative combinations of elementary symmetric functions,
\[
\sin\theta \sum_{k=0}^{[n/2]} (-1)^k \sigma_{2k}(D^2 u) + \cos \theta \sum_{k=0}^{[(n-1)/2]}(-1)^k\sigma_{2k+1}(D^2 u) =0.
\]

A complex analogue of this equation on compact K\"ahler manifolds also appeared naturally from the study of Mirror Symmetry, see Leung-Yau-Zaslow \cite{LYZ}, Collins-Jacob-Yau \cite{CJY} and Collins-Yau \cite{CY}. Moreover, equations of the form
\[
\sigma_1 (D^2 u) + b\, \sigma_n(D^2 u) = C
\]
for some constants $b\geq 0$ and $C>0$ also arise from the study of $J$-equation on toric varieties by Collins-Sz\'ekelyhidi \cite{CS}.

\medskip
\par
The above mentioned examples, Fu-Yau equation and special Lagrangian equation, are of a similar form as equation (\ref{equ-s}), but have their special structures which play essential role in the study of their solvability. It would be desirable to establish a general frame work for this type of equations. Indeed, equation of this general type was already considered by Krylov as an important example for the applications of the general notion of fully nonlinear elliptic equations developed in  \cite{Krylov95}. More precisely, Krylov considered the case with $\alpha(x)\leq 0$ and studied Dirichlet problem of the following degenerate equation in a (k-1)-convex domain $D$ in $\mathbb R^n$,
\begin{eqnarray}\label{krylov-equ}
 \sigma_k (D^2 u) = \sum_{l=0}^{k-1} \alpha_{l}(x)\,\sigma_{l}(D^2 u),
\end{eqnarray}
with all the coefficient $\alpha_{l}(x)\geq 0$ for $0\leq l\leq k-1$. In \cite{Krylov95}, Krylov observed that the above equation is elliptic in the admissible $\Gamma_k$ cone. By using certain concave structure of the elliptic operator, he reduced the above equation to the elliptic Bellman's equation and then applied the general theorems on the Bellman equations to obtain the crucial $C^{1, 1}$ a priori estimates provided that $\alpha_l(x)\geq 0$ and $\alpha_l^{1/(k-l+1)}(x) \in C^{1, 1}(D)$ for $0\leq l \leq k-1$.

\medskip

\par
Using the observation in the study of equation (\ref{equ-s}) on spheres, we can also study the corresponding Dirichlet problem over bounded domains in $\mathbb R^n$. In comparison with Krylov's equation, a key new feature is that {\it there is no sign requirement for the coefficient function of $\sigma_{k-1}$}. In fact, by viewing the above mentioned Fu-Yau equation and special Lagrangian equation, it is important to study the equations formed by linear combinations of elementary symmetric functions with possibly alternative signs. However, as we will see in Section 2, the structure and behavior of the equation will be quite different from Krylov's case with this new feature. In general, examples show that one can not expect the existence of solutions in $\Gamma_k$ cone as Krylov's case, which leads us to look for solutions in a larger cone. On the other hand, the admissible set should not be too large as we still need to keep the important ellipticity and concavity properties of the operator. The key observation given in Proposition \ref{cone-concavity} is that the proper admissible set is $\Gamma_{k-1}$ for the new equation. Based on this, we can prove the following $C^{1, 1}$ a priori estimate for the corresponding degenerate problem.

 \begin{theorem}
 Let $u\in C^4(\Omega) \cap C^2(\bar\Omega)$ be an admissible solution, that is $D^2 u \in \Gamma_{k-1}$, of the equation
 \begin{eqnarray}\label{equ-d}
 \sigma_k(D^2 u) + \alpha(x)\, \sigma_{k-1}(D^2 u) = \sum_{\ell=0}^{k-2} \alpha_\ell(x)\, \sigma_\ell(D^2 u)
 \end{eqnarray}
 over a bounded domain $\Omega \subset\mathbb R^n$ with smooth boundary $\partial \Omega$, and $u= \phi$ on $\partial\Omega$ for some given smooth function $\phi(x)$.
 Assume that $\alpha_\ell(x) \geq 0$ for $\ell=0, 1, \cdots, k-2$.
 Then
 \begin{eqnarray}\label{C11est}
  \sup_{\Omega} | D^2 u| \leq C \left( 1 + \sup_{\partial\Omega} |D^2 u|\right),
   \end{eqnarray}
where $C$ is a constant depending on $\|\alpha_{\ell}^{1/(k-\ell)}\|_{C^{1, 1}}$, $\|\alpha(x) \|_{C^{1, 1}}$, $\|u\|_{C^0}$ and $\Omega$, but independent of the lower bound $\inf_{\Omega} \alpha_{\ell}(x)$ with $\ell=0, \cdots, k-2$.
 \end{theorem}

\medskip

There is an interesting PDE problem in finding the optimal power $p_{\ell}$ in the requirement of $\alpha_\ell^{1/ p_\ell}\in C^{1, 1}$ to obtain the above $C^{1, 1}$ estimate for the degenerate fully nonlinear equations. For Krylov's equation (\ref{krylov-equ}), $C^{1, 1}$ estimates were proved with the assumption that $p_l\geq k-l +1$ for $0\leq l \leq k-1$, and this was weakened to be $p_l \geq k-l$ by Dong \cite{Dong} and hence partially confirmed a question proposed by Krylov in \cite{Krylov95}. Here, in the above theorem, no assumption is required on the coefficient of $\sigma_{k-1}$, except $\alpha(x) \in C^{1, 1}$. In particular, we do not require it to be non-negative. For $\alpha_{\ell}$ with $0\leq \ell \leq k-2$, we assume that $\alpha_\ell^{1/ p_\ell}\in C^{1, 1}$ with $p_\ell \geq k-\ell$. But it is still an interesting question whether the estimate still holds if one assumes $p_\ell \geq k-\ell-1$ for all $1\leq\ell \leq k-2$. In fact, one expects $p_\ell = k-\ell -1$ to be the optimal power to guarantee the $C^{1, 1}$ estimate. For the case of Monge-Amp\`ere equation, that is $k=n$ and $\alpha_\ell=0$ for $1\leq \ell\leq n-1$, this was proved by Guan-Trudinger-Wang \cite{GTW} and the power $n-1$ was also shown to be sharp by an example in \cite{Wang95}. For the Hessian equation, that is $\alpha_\ell=0$ for $1\leq \ell \leq k-1$ with $k<n$, estimate as (\ref{C11est}) was obtained in \cite{ITW} with power $k-1$ and a modified example shows that $k-1$ is the lowest possible power in \cite{DPZ}. An interesting question left open in \cite{ITW} is to establish the boundary $C^{1, 1}$ estimate for the solution of the Hessian equation under the optimal condition.

\medskip

\par
To obtain the global $C^{1, 1}$ estimate for equation (\ref{equ-d}), it is enough to derive the estimates on the boundary $\partial\Omega$. Here, we follow the idea by Guan \cite{GuanB1, GuanB2} and obtain the boundary estimate under the sub-solution assumption.
\begin{theorem}
Let $u\in C^4(\Omega) \cap C^2(\bar\Omega)$ be an admissible solution of equation (\ref{equ-d}) with $u= \phi$ on the boundary. Suppose $\alpha_\ell>0 \in C^{1, 1}(\bar\Omega)$ and $\alpha(x) \in C^{1, 1}(\bar \Omega)$.
If there exists an admissible sub-solution $\underline u \in C^2(\bar \Omega)$:
\begin{equation}\label{subsolution}
\begin{cases}
\sigma_k(D^2 \underline u) + \alpha(x) \,\sigma_{k-1}(D^2 \underline u) \geq \sum_{\ell=0}^{k-2} \alpha_\ell(x)\, \sigma_\ell(D^2 \underline u) & \text{ in } \Omega\\
\underline u = \phi & \text{ on } \partial \Omega.
\end{cases}
\end{equation}
Then, there exists a constant $C$ depending on $\|u\|_{C^1(\bar\Omega)}$, $\|\underline u\|_{C^2(\bar\Omega)}$, $\|\alpha\|_{C^{1, 1}(\bar\Omega)}$ and $\|\alpha_\ell\|_{C^{1, 1}(\bar\Omega)}$ such that
\[
\max_{\partial\Omega} |D^2 u|\leq C
\]
\end{theorem}

\medskip

The paper is organized as follows. In \S 2, we provide some background and discuss the ellipticity and concavity of equation (\ref{equ-s}). In \S 3, we prove the a priori estimates for the equation on $\mathbb S^n$ and show the existence of the solution by using degree method. In the last part of \S 3, we will also derive the $C^{1, 1}$ estimates on $\mathbb S^n$ for the degenerate case. In \S 4, we study the Dirichlet problem of equation (\ref{equ-d}) by deriving the interior and boundary $C^{1, 1}$ estimates.

\

{\bf Acknowledgements}. It is our pleasure to dedicate this work to Professor Duong H. Phong on the occasion of his 65th birthday, in honor of his influence in Mathematics. We wish to thank the referees for useful comments and suggestions.

 \section{Ellipticity and concavity of the equation}

 \par
 In this section, we study the ellipticity and concavity of equation
 \begin{eqnarray}\label{equationD}
 \sigma_k(D^2 u) + \alpha(x) \,\sigma_{k-1}(D^2 u) = \sum_{\ell=0}^{k-2} \alpha_{\ell}(x)\, \sigma_{\ell}(D^2 u)
 \end{eqnarray}
 where $\alpha_{\ell}(x)\in C^{\infty}(\Omega)$, $0\leq \ell\leq k-2$, are non-negative functions with $\sum_{\ell=0}^{k-2} \alpha_{\ell}(x)>0$.
 Those two properties play fundamental role when we solve the equation by the so-called method of continuity, according to the general PDE theories.

\medskip

 \par
In the landmark work \cite{CNSI, CNSIII}, Caffarelli-Nirenberg-Spruck investigated the Hessian equations systematically,
 \begin{eqnarray}\label{Hessian equation}
 \sigma_k (D^2 u) = f(x), \ \textit{ for } x\in \Omega\subset \mathbb R^n.
 \end{eqnarray}
$u$ is the {\it admissible solution} of equation (\ref{Hessian equation}) if $\lambda(D^2 u) \in \Gamma_k$ cone with
\begin{eqnarray*}\label{Gammak}
\Gamma_k = \{ \ \lambda\in \mathbb R^n \ | \ \sigma_1(\lambda)>0, \cdots, \sigma_{k}(\lambda)>0 \}.
\end{eqnarray*}
A class of more general Hessian type equations was considered by Krylov \cite{Krylov95}. In particular, he observed that the natural admissible cone to make equation
 \begin{align}\label{Krylov equation}
 \sigma_k (D^2 u) = \sum_{l=0}^{k-1} \alpha_{l}(x) \sigma_{l} (D^2 u), \ \textit{\rm with } \alpha_{l}(x) >0, \,  l=0, 1, \cdots, k-1
 \end{align}
 elliptic is also the $\Gamma_k$-cone, which is the same as the Hessian equation case.

 \medskip

\par
Comparing equation \eqref{equationD} with \eqref{Hessian equation} or \eqref{Krylov equation}, it is plausible to guess that the {\it admissible cone} should also be $\Gamma_k$. However, the following simple example indicates that, in general, this is not necessarily true.
 \par
 \begin{example}
{\rm Consider equation
 \begin{eqnarray}\label{11}
 \sigma_2(D^2u) +\sigma_1(D^2 u) = f(x), \ x\in \Omega \subset\mathbb R^n,
 \end{eqnarray}
 with $f(x) \geq 0$ smooth. Suppose $u$ is an admissible solution of the above equation. Let $ v = u + \frac{1}{4}|x|^2$, then
 $D^2v= D^2 u+ \frac{1}{2} I_2$.
  Therefore, $u$ is a solution of equation (\ref{11}) if and only if $v$ solves the following equation
 \begin{eqnarray}\label{12}
 \sigma_2 (D^2 v) = f(x) + \frac{1}{4}\sigma_2(I_2):=g(x)>0.
 \end{eqnarray}
According to the solvability of standard Hessian equation in \cite{CNSIII}, the admissible cone for
equation (\ref{12}) is $\Gamma_2$. This indicates that one should seek the solution for equation (\ref{11}) in a convex set containing $\Gamma:=\{\lambda\in \mathbb R^n \, |\,  \lambda+ {1\over 2} I_n \in \Gamma_2\}$. On the other hand, it is easy to check that $\Gamma_2$-cone is properly contained in the set $\Gamma$. Therefore, we should looking for admissible solutions of equation (\ref{11}) in a convex set which is larger than $\Gamma_2$.
\par
In fact, one can easily verify that $u(x_1, x_2) = -{1\over 8}x_1^2 +{1\over 2} x_2^2$ solves equation (\ref{11}) with $f(x) \equiv {1\over 2}$ in $\mathbb R^2$. We compute that $\sigma_1(D^2u) = {3\over 4}>0$, but $\sigma_2(D^2 u) = -{1\over 4} <0$.

\medskip

There are also similar examples on spheres. It's known (\cite{Alex1}) that there exist plenty of functions $u\in C^4(\mathbb S^2)$ such that  $\sigma_1(W_u) = \Delta u+ 2 u >0$ on $\mathbb S^2$ while $\sigma_2 (W_u)(x_0)<0$ for some point $x_0\in \mathbb S^2$. Therefore, $u$ satisfies equation
\begin{eqnarray}
 \sigma_2(W_u) + A \, \sigma_1(W_u) = f(x) \ \ {\rm  on }\ \ \mathbb S^2.
\end{eqnarray}
for some $f(x)>0$ and large constant $A>0$. It's easy to show that such solutions are unique up to linear combination of coordinate functions $x_1, x_{2}, x_3$. }

\end{example}
\

This simple example suggests that we should look for admissible solution for equation (\ref{equationD}) in a set larger than $\Gamma_k$ cone. However, it should not be too large as we still need to keep the important ellipticity and concavity properties of the operator. The following proposition says that the proper admissible set is $\Gamma_{k-1}$ for the new equation.

\medskip

 \begin{proposition}\label{cone-concavity}
If $u$ is a smooth function with $\lambda(D^2 u) \in \Gamma_{k-1}$, then the operator
 \begin{eqnarray}\label{operatorG}
  G(\lambda(D^2u)) := \frac{\sigma_k(\lambda(D^2u))}{\sigma_{k-1}(\lambda(D^2u))} - \sum_{\ell=0}^{k-2} \alpha_{\ell}(x) \frac{\sigma_{\ell}(\lambda(D^2u))}{\sigma_{k-1}(\lambda(D^2u))}
 \end{eqnarray}
 is elliptic and concave.
 \end{proposition}
 \begin{proof}
 To check the ellipticity, it suffices to show that
 \begin{eqnarray*}
 G^{ii} = \frac{\partial G(\lambda)}{\partial\lambda_i} >0, \ \text{ for } \lambda \in \Gamma_{k-1}.
 \end{eqnarray*}
 By direct computation, it follows that
 \begin{eqnarray*}
 G^{ii} &=& \left(\frac{\sigma_k }{\sigma_{k-1}}\right)^{ii} - \sum_{\ell=0}^{k-2} \alpha_{\ell}(x) \left(\frac{\sigma_{\ell} }{\sigma_{k-1}}\right)^{ii}\\
 &=& \frac{ \sigma_{k-1}(\lambda | i) \sigma_{k-1} - \sigma_k \sigma_{k-2}(\lambda | i)}{\sigma_{k-1}^2} -  \sum_{\ell=0}^{k-2} \alpha_{\ell}(x) \frac{ \sigma_{\ell-1}(\lambda | i) \sigma_{k-1} - \sigma_{\ell} \sigma_{k-2}(\lambda | i)}{\sigma_{k-1}^2}.
 \end{eqnarray*}
 Note that, for any $i, \ell = 1, \cdots, n$, we have $\sigma_\ell(\lambda) = \sigma_\ell (\lambda | i )+ \lambda_i \sigma_{\ell-1} (\lambda | i) $. Thus,
 \begin{eqnarray*}
  \sigma_{k-1}(\lambda | i) \sigma_{k-1} - \sigma_k \sigma_{k-2}(\lambda | i)= \sigma_{k-1}^2(\lambda | i) - \sigma_{k}(\lambda | i) \sigma_{k-2}(\lambda | i) \geq\frac{n}{k(n-k+1)}\sigma_{k-1}^2(\lambda | i)
 \end{eqnarray*}
 by Newton-Maclaurin inequality. For more detailed computation, please see \cite{Guan}. Similarly, we can also obtain that
 \[
 \sigma_{\ell-1}(\lambda | i) \,\sigma_{k-1}(\lambda) - \sigma_{\ell}(\lambda)\, \sigma_{k-2}(\lambda | i) < 0.
 \]
This implies that $G^{ii} \geq 0$ if $\alpha_{\ell}(x) \geq 0$.

\medskip

\par
The concavity of operator $G(D^2 u)$ follows from the concavity of operator $\frac{\sigma_k}{\sigma_{k-1}}$ and convexity of operator $\frac{\sigma_{\ell}}{\sigma_{k-1}}$ for any $\ell =0, 1, \cdots, k-2$.
\par
In fact, it is known that $\left( {\sigma_q\over \sigma_p}\right)^{1/(q-p)}$ is concave in the $\Gamma_q$ cone for any $0\leq p < q \leq n$, see for example, the last section in \cite{Krylov95}. 
To see the convexity of ${\sigma_\ell\over \sigma_{k-1}}$ for $0\leq \ell \leq k-2$, we simply use the concavity of $F:= \left( {\sigma_{k-1}\over \sigma_\ell}\right)^{1/(k-1-\ell)}$. We can write
$ {\sigma_\ell \over \sigma_{k-1}} = F^{-(k-1-\ell)}$,
and compute
\[
 \left({\sigma_\ell \over \sigma_{k-1}}\right)^{ij, pq} = (k-\ell)\, (k-\ell -1) \, F^{-(k-\ell +1)} \, F^{ij} \, F^{pq} - (k-\ell -1) \, F^{-(k-\ell)} \, F^{ij, pq}
 \]
The second term is negative because of the concavity of $F$. On the other hand, it is a well-known fact proved by Huisken-Sinestrari \cite{HS} (see also \cite{LRW}) that ${\sigma_k\over \sigma_{k-1}}$ is concave in the $\Gamma_{k-1}$ cone. Therefore, it follows that the operator $G$ is concave in $\Gamma_{k-1}$.

 \end{proof}

\

For the convenience of notations, we will denote
\begin{eqnarray}\label{operatorG}
G_k(\lambda) ={\sigma_k\over \sigma_{k-1}}(\lambda), \ \  \ G_{\ell}(\lambda) =-{\sigma_\ell\over \sigma_{k-1}}(\lambda) \textit{ \rm for } 0\leq \ell \leq k-2
\end{eqnarray}
Under this notation, for $u$ in the $\Gamma_{k-1}$ admissible cone, equation (\ref{equ-s}) is equivalent to
\begin{eqnarray}\label{G-op}
G(W_u(x)) := G_k(W_u(x))  + \sum_{\ell=0}^{k-2} \, \alpha_{\ell}(x) \, G_{\ell}(W_u(x)) = -\alpha(x).
\end{eqnarray}
 Moreover, it follows from the above proposition, we know the operator $G$ is elliptic and concave in $\Gamma_{k-1}$ cone. We will make essential use of these properties in the following sections.

\medskip

We remark that the ellipticity and concavity of the operator proved in Proposition \ref{cone-concavity} purely rely on the algebraic properties of the elementary symmetric functions. Therefore, Proposition \ref{cone-concavity} still holds if we replace $D^2 u$ by the spherical hessian matrix $W_u = u_{ij} + u \, \delta_{ij}$.

 \

 \section{Existence of admissible solutions on spheres}

 \par
 In this section, we will first derive the a priori estimates for the admissible solution of equation (\ref{equ-s}) on $\mathbb S^n$, that is
 \begin{eqnarray}\label{equ-s1}
\sigma_k(W_u(x))+ \alpha(x)\sigma_{k-1}(W_u(x)) =\sum_{\ell=0} ^{k-2} \alpha_{\ell}(x) \, \sigma_{\ell}(W_u(x)), \  \ x\in \mathbb S^n
\end{eqnarray}
and then prove the existence of the solution via the degree method.
\medskip
\par We note that for any solution $u(x)$ of (\ref{equ-s1}), $u(x)+ l(x)$ is also a solution of the equation for any linear function $l(x) = \sum_{i=1}^{n+1} a_i \, x_i$. So, we will confine ourselves to solutions satisfying the following orthogonal condition
\begin{eqnarray}\label{orth}
\int_{\mathbb S^n} x_i \, u \, dx =0 \ \ \textit{ for } \forall \, i=1, \cdots, n+1.
\end{eqnarray}

\

 \subsection{The a priori estimates}$\\$

 \par
 In this sub-section, we proceed to derive {\it a priori} estimates in the non-degenerate case. We assume $u\in C^4(\mathbb S^n)$ is an admissible solution of equation (\ref{equ-s1}), i.e.
 \[
 W_u = u_{ij} + u\, \delta_{ij}\in \Gamma_{k-1}.
 \]

\medskip

\subsubsection{$C^2$ estimate} $\\$

We will derive the a priori $C^2$ estimate first, which are independent of the group invariance assumption.

\begin{proposition}\label{c2-non-deg}
Let $u\in C^4(S^n)$ be an admissible solution of equation (\ref{equ-s1}) with $\alpha(x)\in C^{1, 1}(S^n)$, and assume that $\alpha_\ell(x)>0\in C^{1, 1}(S^n)$.
Then, there exists a positive constant $C$ depending on $n$, $\|\alpha(x)\|_{C^{1, 1}(S^n)}$, $\|\alpha_\ell(x)\|_{C^{1,1}(S^n)}$ and $\inf_{S^n} \alpha_\ell(x)$ such that
\begin{eqnarray}\label{c2-est}
0< \lambda_{\max}(x) \leq  C \ \textit{ on } \mathbb S^n
\end{eqnarray}
where $\lambda_{\max}(x)$ is the largest eigenvalue of the matrix $W_u(x) = u_{ij} + u\, \delta_{ij}$. In particular, for any eigenvalue $\lambda_i(x)$ of $W_u(x) = u_{ij} + u\, \delta_{ij}$, we have
\begin{eqnarray}\label{c2-other}
\left|\lambda_i(x)\right| \leq (n-1) \, C \ \textit{ on } \mathbb S^n.
\end{eqnarray}
\end{proposition}

\medskip

\par
\noindent
{\it Proof:} We first note that the positivity of $\lambda_{\rm max}(x)$ follows from the fact that $u$ is an admissible solution, that is $W_u(x) \in \Gamma_{k-1}$. In particular, $W_u(x) \in \Gamma_1$. Therefore,
\[
0< \sigma_1(W_u) \leq n \, \lambda_{\max}.
\]
\par
To obtain the upper bound for $\lambda_{\max}$, we will consider the auxiliary function
\[
H= {\rm tr}(W_u) = \Delta u+ nu.
\]
Note that, again since $W_u\in \Gamma_{k-1}\subset \Gamma_1$, $\lambda_{\max} \leq \sigma_1(W_u) = H$. Then, it will be sufficient to derive an upper bound for $H$. Suppose that the maximum value of $H$ is attained at some point $x_0\in \mathbb S^n$. We may choose an orthonormal local frame $e_1, e_2, \cdots, e_n$ near $x_0$ such that $u_{ij}(x_0)$ is diagonal. For the standard metric on $\mathbb S^n$, one may easily check the commutator identity
\[
H_{ii} =\Delta W_{ii} - nW_{ii} +H,
\]
where $W_{ii}= u_{ii}+ u$.
\par
Since $u\in C^4(S^n)$ is an admissible solution, that is $W_u\in \Gamma_{k-1}$, it follows from the observation in Proposition \ref{cone-concavity} that $(G^{ij})$ is positive definite.
Moreover, at the maximal point $x_0$, $G^{ij}=  G_k^{ij}+ \sum_{\ell=0}^{k-2} \alpha_\ell\, G_\ell^{ij}$
is diagonal. By the above commutator identity, it follows that, at $x_0$,
\begin{eqnarray}\label{3k}
0\geq G^{ij}H_{ij} = G^{ii}\left(\Delta W_{ii}\right) - n G^{ii}W_{ii} + H \sum_{i=1}^n G^{ii}.
\end{eqnarray}
\noindent
Since $G_k$ is homogeneous of degree one and $G_\ell$ is homogeneous of degree $-(k-1-\ell)$ for $0\leq \ell \leq k-2$, we can compute
\begin{eqnarray}\label{4k}
G^{ii}W_{ii} &=& G_k-\sum_{\ell=0} ^{k-2}(k-1-\ell)\,\alpha_\ell\,G_\ell= -\alpha+ \sum_{\ell=0}^{k-2} (k-\ell) \alpha_\ell \, \frac{\sigma_\ell(W_u)}{\sigma_{k-1}(W_u)}.
\end{eqnarray}

Next we differentiate equation (\ref{equ-s1}) and obtain
\begin{eqnarray*}
G_k^{ij} W_{ijp}+ \sum_{\ell=0}^{k-2}\alpha_\ell\, G_\ell^{ij}W_{ijp} + \sum_{\ell=0}^{k-2} (\alpha_\ell)_p\, G_\ell=-(\alpha)_p,
\end{eqnarray*}
where $(\alpha_\ell)_p$ denotes the derivative of function $\alpha_\ell(x)$. Differentiate the equation another time and we obtain
\begin{eqnarray*}
 &&\sum_{p=1}^nG_k^{ij, rs} W_{ijp}W_{rsp} + G_k^{ij}  \,  \Delta W_{ij}
  +\sum_{p=1}^n \sum_{\ell=0}^{k-2} \alpha_\ell  \, G_{\ell}^{ij, rs} W_{ijp}W_{rsp}+ \sum_{p=1}^n\sum_{\ell=0}^{k-2} \alpha_\ell \, G_\ell^{ij}\Delta W_{ij} \\
  &&+ 2 \sum_{p=1}^n\sum_{\ell=0}^{k-2} (\alpha_\ell)_p \,  G_\ell^{ij} W_{ijp} + \sum_{\ell=0}^{k-2} \Delta(\alpha_\ell)  \, G_\ell =-\Delta \alpha.
\end{eqnarray*}
The above identities imply
\begin{eqnarray}\label{5k}
G^{ij} \Delta W_{ij} &=& - \sum_{p=1}^nG_k^{ij, rs} W_{ijp}W_{rsp}
- \sum_{p=1}^n\sum_{\ell=0}^{k-2} \alpha_\ell \, G_\ell^{ij, rs} W_{ijp}W_{rsp} \\\nonumber
&& - 2\sum_{p=1}^n\sum_{\ell=0}^{k-2} (\alpha_\ell)_p\, G_\ell^{ij} W_{ijp}- \sum_{\ell=0}^{k-2} \Delta(\alpha_\ell) \, G_\ell -\Delta \alpha.
\end{eqnarray}
Putting (\ref{5k}) and (\ref{3k}) into (\ref{4k}) and using the concavity of $G_k= \frac{\sigma_k}{\sigma_{k-1}}$, that is
\[
\sum_{ij, rs} \left(G_k\right){}^{ij, rs}  X_{ij} \,  X_{rs} \leq 0, \ \ \ \textit{\rm for any symmetric matrix } (X_{ij})\in \mathbb R^{n\times n},
\]
we obtain
\begin{eqnarray}\label{6k}
0\geq G^{ij} H_{ij} &\geq &
-\sum_{p=1}^n\sum_{\ell=0}^{k-2} \alpha_\ell \,G_\ell^{ij, rs} W_{ijp}W_{rsp} - 2 \sum_{p=1}^n\sum_{\ell=0}^{k-2} (\alpha_\ell)_p \, G_\ell^{ij} W_{ijp}\\\nonumber
&&
+ \sum_{\ell=0}^{k-2} \left(n(k-\ell)\alpha_\ell - \Delta\alpha_\ell\right)\, G_\ell
+H\sum_{i=1}^n G^{ii} + n\alpha - \Delta\alpha
\end{eqnarray}

\smallskip

Next, we need to estimate the right hand side of the above inequality. Essentially, we need to get a lower bound for the following terms
\begin{eqnarray}\label{est-concave2}
-\sum_{p=1}^n \alpha_\ell\, G_\ell^{ij, rs} W_{ijp}W_{rsp} - 2 \sum_{p=1}^n(\alpha_\ell)_p \,G_\ell^{ij} W_{ijp}, \ \ {\rm for } \ \ell=0, 1, \cdots, k-2.
\end{eqnarray}

It was shown by Krylov in \cite{Krylov95} that the operator $\left( \frac{\sigma_{k-1}}{\sigma_\ell}\right)^{\frac{1}{k-1-\ell}}$ is concave for $0\leq \ell\leq k-2$. It follows that
\[
\left( - \frac{1}{G_\ell}\right)^{\frac{1}{k-1-\ell}} \text{ is a concave operator for } \forall \, \ell=0, 1, \cdots, k-2.
\]
 Then, for any symmetric matrix $ (X_{ij})\in \mathbb R^{n\times n}$, we have
\[
\sum_{ij, rs} \left\{ \left( - \frac{1}{G_\ell}\right)^{\frac{1}{k-1-\ell}}\right\}^{ij, rs} X_{ij} X_{rs}\leq 0.
\]
By direct computation, we have
\[
\left\{\left (-\frac{1}{G_\ell}\right)^{\frac{1}{k-1-\ell}+1}  G_\ell^{ij, rs} + \left(1+\frac{1}{k-1-\ell}\right) \left( - \frac{1}{G_\ell}\right)^{\frac{1}{k-1-\ell}+2}G_\ell^{ij}\,G_\ell^{rs}\right\} X_{ij}X_{rs} \leq 0.
\]
As a consequence, we have
\begin{eqnarray}\label{concavity}
- G_\ell^{ij, rs} X_{ij}X_{rs} \geq - \left( 1+ \frac{1}{k+1-\ell}\right) G_\ell^{-1}G_\ell^{ij}\,G_\ell^{rs} X_{ij}X_{rs},
\end{eqnarray}
for $0\leq \ell \leq k-2$. Using this inequality from concavity of the operators, we can estimate
\begin{eqnarray*}
&&-\sum_{p=1}^n \alpha_\ell \,G_\ell^{ij, rs} W_{ijp}W_{rsp} - 2 \sum_{p=1}^n(\alpha_\ell)_p \,G_\ell^{ij} W_{ijp}\\
&\geq & -  \left( 1+ \frac{1}{k+1-\ell}\right) \alpha_\ell\, G_\ell^{-1}\,\sum_{p=1}^n G_\ell^{ij}\,G_\ell^{rs} W_{ijp}W_{rsp}- 2 \sum_{p=1}^n (\alpha_\ell)_p\, G_\ell^{ij} W_{ijp}\\
& = & -  \left( 1+ \frac{1}{k+1-\ell}\right) \alpha_\ell\, G_\ell^{-1}\,\sum_{p=1}^n \left(G_\ell^{ij}\,G_\ell^{rs}W_{ijp}W_{rsp} + \frac{2} {1+ \frac{1}{k+1-\ell}} \frac{(\alpha_\ell)_p} {\alpha_\ell}\, G_\ell \, G_\ell^{ij} W_{ijp}\right)\\
&=&  - {k-\ell + 2\over k-\ell +1}  \alpha_\ell\, G_\ell^{-1}\, \sum_{p=1}^n\left(\sum_{i,j} G_\ell^{ij} W_{ijp} + \frac{1} {1+ \frac{1}{k+1-\ell}} \frac{(\alpha_\ell)_p} {\alpha_\ell}\, G_\ell \right)^2+\frac{1} {1+ \frac{1}{k+1-\ell}} \sum_{p=1}^n \frac{(\alpha_\ell)^2_p}{\alpha_\ell} \,G_\ell.
\\
&\geq&  \frac{1} {1+ \frac{1}{k+1-\ell}} \sum_{p=1}^n \frac{(\alpha_\ell)^2_p}{\alpha_\ell} \,G_\ell
\end{eqnarray*}
where we used the fact that $\alpha_\ell> 0$ and $G_\ell <0$ to obtain the last inequality for $0\leq \ell \leq k-2$.

Putting the above estimates back to (\ref{6k}), we obtain,
\begin{eqnarray}\label{7k}
0& \geq& G^{ij}H_{ij}  \geq
  \sum_{\ell=0}^{k-2} \left( n\,(k-l) \alpha_\ell - \Delta\alpha_\ell + \frac{|\nabla\alpha_\ell|^2}{\alpha_\ell}\right) G_\ell + H\sum_{i=1}^n G^{ii} -C\\\nonumber
&\geq & H\sum_{i=1}^n G^{ii} - C_1\sum_{\ell=0}^{k-2} \frac{\sigma_\ell}{\sigma_{k-1}}-C
\end{eqnarray}
where $C_1$ is a positive constant depending on $\|\alpha_\ell\|_{C^{1, 1}(S^n)}$ with $0\leq \ell \leq k-2$. To get the desired estimate, we still need to get a lower bound for $\sum_{i=1}^n G^{ii}$. By the definition of operator $G$ and straightforward computation, we have
\begin{eqnarray*}
\sum_{i=1}^n G^{ii} &=& \sum_{i=1}^n \left(\frac{\sigma_k}{\sigma_{k-1}}  \right)^{ii} - \sum_{i=1}^n\sum_{\ell=0}^{k-2} \alpha_\ell \left( \frac{\sigma_\ell}{\sigma_{k-1}}\right)^{ii}\\
&=& \sum_{i=1}^n \frac{\sigma_{k-1}(\lambda | i)}{\sigma_{k-1}} - \frac{\sigma_k}{\sigma_{k-1}^2} \sum_{i=1}^n \sigma_{k-2}(\lambda | i) + \frac{ \alpha_0}{\sigma_{k-1}^2} \sum_{i=1}^n \sigma_{k-2}(\lambda | i)\\
&&+\sum_{\ell=1}^{k-2}\alpha_\ell\frac{\sum_i \sigma_\ell\sigma_{k-2}(\lambda | i) - \sum_i\sigma_{k-1} \sigma_{\ell-1}(\lambda | i)}{\sigma_{k-1}^2}
\\& = & (n-k+1) - (n-k+2)\frac{\sigma_k \sigma_{k-2}}{\sigma_{k-1}^2} + (n-k+2)  \alpha_0\frac{\sigma_{k-2}}{\sigma_{k-1}^2} \\
 &&+ \sum_{\ell=1}^{k-2} \alpha_\ell \frac{ (n-k+2) \sigma_\ell \sigma_{k-2} - (n-\ell+1) \sigma_{k-1}\sigma_{\ell-1} }{\sigma_{k-1}^2}\\
\end{eqnarray*}
It follows from the Newton-MacLaurin inequality that
\begin{eqnarray*}
\frac{\sigma_k\, \sigma_{k-2}}{\sigma_{k-1}^2} \leq \frac{k-1}{k}\frac{n-k+1}{n-k+2}, \, \textit{\rm and } \  \frac{\sigma_{k-1}\,\sigma_{\ell-1}}{\sigma_\ell \sigma_{k-2}}\leq \frac{n-k+2}{n-\ell+1} \ \textit{ for } 1\leq \ell \leq k.
\end{eqnarray*}
Therefore, we have
\begin{eqnarray*}
\sum_{i=1}^n G^{ii} \geq \frac{n-k+1}{k}.
\end{eqnarray*}
Back to (\ref{7k}), we obtain
\begin{eqnarray}\label{8k}
0 \geq \frac{n-k+1}{k} H - C_1\sum_{\ell=0}^{k-2} \frac{\sigma_\ell}{\sigma_{k-1}}-C
\end{eqnarray}
To get the desired upper bound for $H$, it is enough to show that
\[
\frac{\sigma_\ell}{\sigma_{k-1}} \leq CH^{t} \ \ \textit{\rm for } \forall \, \ell=0, 1, \cdots, k-2
\]
for some uniform constant $C$ and $t<1$.

\medskip

Now, we are in the place to use the non-degeneracy condition in Proposition \ref{c2-non-deg}. Without loss of generality, we may assume that
\[
\alpha_m(x)\geq c_0>0
\]
 where $0\leq m\leq k-2$ is the largest integer such that $\alpha_\ell(x)\neq 0$. Under this assumption, the equation under consideration becomes
\begin{eqnarray}\label{9k}
\sigma_{k}(W_u) + \alpha(x) \sigma_{k-1}(W_u) = \sum_{\ell=0}^{m}\alpha_{\ell}(x)\,\sigma_\ell(W_u),
\end{eqnarray}
where $m\leq k-2$ and $\alpha_{m}(x) \geq c_0>0$ and $\alpha_l \geq 0$ for $0\leq l \leq m-1$. We would like to bound $\frac{\sigma_{l}}{\sigma_{k-1}}$ from above for any $0\leq l\leq m$.

\smallskip

{\it Step 1: Using the non-degenerate assumption $\alpha_m(x)\geq c_0>0$ to control the leading term}
\[
\frac{\sigma_m}{\sigma_{k-1}}(W_u(x_0)).
\]
 Note that, by equation (\ref{9k}),
\[
\sigma_k + \alpha(x) \sigma_{k-1} \geq \alpha_m(x) \sigma_{m}\geq c_0 \,\sigma_m.
\]
Recall that $W_u\in \Gamma_{k-1}$ and therefore, we have either $\sigma_k(W_u)\geq 0$ or $\sigma_k(W_u)\leq 0$ at point $x_0\in \mathbb S^n$. If $\sigma_k(W_u(x_0)) \leq 0$, then we are done since the above inequality implies
\[
\frac{\sigma_m}{\sigma_{k-1}}\leq | \alpha(x_0) |.
\]
Now, we assume that $\sigma_{k}(W_u(x_0))\geq 0$. We will discuss into two cases.
\begin{itemize}
 \item[{\it Case 1.}] If $\frac{\sigma_k}{\sigma_{k-1}}\leq C H^{\frac{1}{k}}$, then we get
 \begin{eqnarray}\label{10k}
 \frac{\sigma_m}{\sigma_{k-1}} \leq \frac{1}{c_0} \left( \frac{\sigma_k}{\sigma_{k-1}} + \alpha(x) \right) \leq C H^{\frac{1}{k}}.
 \end{eqnarray}

 \medskip

 \item[{\it Case 2.}] If $\frac{\sigma_k}{\sigma_{k-1}}\geq C H^{\frac{1}{k}}$, i.e., $\frac{\sigma_{k-1}}{\sigma_{k}}\leq C H^{-\frac{1}{k}}$. Note that
\[
\frac{\sigma_m}{\sigma_{k-1}}= \frac{\sigma_m}{\sigma_{m+1}}\, \cdot\frac{\sigma_{m+1}}{\sigma_{m+2}}\,\cdots\, \frac{\sigma_{k-2}}{\sigma_{k-1}}.
\]

 Then, it follows from the Newton-MacLaurin's inequality that
\begin{eqnarray}\label{11k}
\frac{\sigma_m}{\sigma_{k-1}} \leq C\left(\frac{\sigma_{k-1}}{\sigma_{k}} \right)^{k-1-m} \leq C H^{-\frac{k-1-m}{k}}.
\end{eqnarray}
\end{itemize}

\medskip

\par
{\it Step 2:  Estimate}
\[
\frac{\sigma_\ell}{\sigma_{k-1}}(W_u(x_0)) \ \textit{\rm for } \, 0\leq \ell <m.
\]
We observe that
\[
\frac{\sigma_m}{\sigma_{k-1}}= \frac{\sigma_m}{\sigma_{m+1}}\,\frac{\sigma_{m+1}}{\sigma_{m+2}}\,\cdots \,\frac{\sigma_{k-2}}{\sigma_{k-1}} \geq C\left(\frac{\sigma_{m}}{\sigma_{m+1}} \right)^{k-1-m}.
\]
and then
\[
\frac{\sigma_{\ell}}{\sigma_{\ell+1}} \leq \frac{\sigma_{l+1}}{\sigma_{l+2}} \leq \cdots \leq \frac{\sigma_{m-1}}{\sigma_{m}} \leq \frac{\sigma_{m}}{\sigma_{m+1}} \leq \left(\frac{\sigma_m}{\sigma_{k-1}} \right)^{\frac{1}{k-1-m}}.
\]
On the other hand,
\[
\frac{\sigma_{\ell}}{\sigma_{k-1}} = \frac{\sigma_{\ell}}{\sigma_{\ell+1}} \cdots \frac{\sigma_{m-1}}{\sigma_{m}}\cdot \frac{\sigma_{m}}{\sigma_{k-1}}
\]
and this implies
\[
\frac{\sigma_{\ell}}{\sigma_{k-1}} \leq \left(\frac{\sigma_m}{\sigma_{m+1}} \right)^{m-\ell} \cdot \frac{\sigma_m}{\sigma_{k-1}} \leq \left(\frac{\sigma_m}{\sigma_{k-1}} \right)^{1+ \frac{m-\ell}{k-1-\ell}}.
\]
By making use of the estimate (\ref{10k}) and (\ref{11k}), we have, in Case 1,
\begin{eqnarray}\label{12k}
\frac{\sigma_{\ell}}{\sigma_{k-1}} \leq C H^{\frac{1}{k} \left( 1+\frac{m-\ell}{k-1-\ell} \right)};
\end{eqnarray}
and in Case 2,
\begin{eqnarray}\label{13k}
\frac{\sigma_{\ell}}{\sigma_{k-1}} \leq C H^{-\frac{k-1-m}{k} \left( 1+\frac{m-\ell}{k-1-\ell} \right)}, \ \ \ {\rm for } \ \ 0\leq \ell <m.
\end{eqnarray}

Putting (\ref{10k}),  (\ref{11k}),  (\ref{12k}) and  (\ref{13k}) into  (\ref{8k}), we see that, in both cases, the first term $ \frac{n-k+1}{k} H$ is the higher order term. Therefore,
\[
H(x_0)\leq C.
\]
This proves the estimate for $\lambda_{\max}$.

\par
To derive the estimate (\ref{c2-other}), we again use the fact that $\sigma_1(W_u)>0$ since $W_u\in \Gamma_{k-1} \subset \Gamma_1$. Suppose $\lambda_1$ is the smallest eigenvalue of $W_u(x)$. Then, it follows that
\[
\lambda_1 \geq - \lambda_2 - \lambda_3 - \cdots - \lambda_n \geq -(n-1) \, \lambda_{\max} \geq -(n-1)\, C.
\]
Therefore, we have the desired estimates for $\lambda_i(x)$.

\qed

\

\par
Next, to obtain the a priori estimates up to $C^{2, \gamma}$ for some $0<\gamma<1$, it is enough to derive the a priori $C^2$ estimate for $u$. Then the H\"older estimate for the $D^2 u$ will follow from the Evans-Krylov theorem for concave and uniform elliptic PDEs. The the bootstrapping argument by applying Schauder estimates repeatedly will lead to the estimate on $C^{m+2, \gamma}$ norm as claimed (\ref{est-m2}) in the main theorem.

\

\subsubsection{$C^0$ estimate} $\\$

In viewing of the estimate obtained in Proposition \ref{c2-non-deg}, the estimate on $\|D^2u\|_{L^\infty(S^n)}$ will follow from the $C^0$ estimate of $u$. In general, for the prescribing curvature problem \cite{Guan-Guan} and the Christoffel-Minkowski problem \cite{Guan-Ma}, the $C^0$ estimates are obtained by using Cheng-Yau's argument \cite{Cheng-Yau} which makes essential use of the quermassintegral inequalities for convex hypersurfaces (that is $W_u \in \Gamma_n$). However, such inequalities are not available for non-convex case as we are working for $W_u\in \Gamma_{k-1}$. So, we will adapt the alternative argument provided in \cite{Guan-Guan}.

\begin{proposition}\label{c0-est}
Suppose $u\in C^4(\mathbb S^n)$ is an admissible solution of equation (\ref{equ-s1}) and $u$ satisfies (\ref{orth}). Then, there exists a positive constant $C$ depending on $n, k$, $\|\alpha(x)\|_{C^{1, 1}(S^n)}$, $\|\alpha_\ell(x)\|_{C^{1,1}(S^n)}$ and $\inf_{S^n} \alpha_\ell(x)$ such that
\[
\|u\|_{C^0(\mathbb S^n)} \leq C.
\]
\end{proposition}
\medskip
\begin{proof}
We will prove the estimate by a blowup argument. Suppose there is no such bound, then there exist $u^m$, $m=1, 2, \cdots $ satisfying equation (\ref{equ-s1}), a constant $\tilde{C}$ independent of $m$ and
\[
\sigma_k(W_{u^m}(x))+ \alpha^m(x)\sigma_{k-1}(W_{u^m}(x)) =\sum_{\ell=0} ^{k-2} \alpha^m_{\ell}(x) \, \sigma_{\ell}(W_{u^m}(x)), \  \ x\in \mathbb S^n
\]
with $\alpha^m(x)$ and $\alpha^m_{\ell}(x)\geq 0$ satisfying
\[
\|\alpha^m\|_{C^2(\mathbb S^n)} \leq \tilde C, \ \ \|\alpha^m_{\ell}\|_{C^2(\mathbb S^n)} \leq \tilde C, \  \ \sum_{\ell=0}^{k-2} \alpha^m_\ell(x)>0
\]
but \[
\|u^m\|_{L^\infty(\mathbb S^n)} \geq m.
\]

\medskip
\par
Let $v^m = {u^m\over \|u^m\|_{L^\infty(\mathbb S^n)}}$. Then, for any $m=1, 2, \cdots$
\begin{eqnarray}\label{unif-c0}
\|v^m\|_{L^\infty(\mathbb S^n)} =1.
\end{eqnarray}
By our previous estimates in Proposition \ref{c2-non-deg}, we have for any eigenvalues $\lambda_i(W_{u^m}(x))$ of $W_{u^m}(x)$,
\[
\left| \lambda_i(W_{u^m}(x)) \right| \leq (n-1) \, \lambda_{\max}(W_{u^m}(x)) \leq C.
\]
It is important to note that the constant $C$ here is independent of $m$. This implies that $v^m$ satisfies the following estimates
\begin{eqnarray}\label{unif-c2}
\left| \lambda_i(W_{v^m}(x)) \right| \leq (n-1) \, \lambda_{\max}(W_{v^m}(x)) \leq {C \over \|u^m\|_{L^\infty(\mathbb S^n)}} \longrightarrow 0.
\end{eqnarray}
In particular, we obtain
\[
\Delta v^m + n \, v^m  \longrightarrow 0.
\]
\smallskip

\par
On the other hand, for a general function $w\in C^4(\mathbb S^n)$ satisfying $W_w = w_{ij} + w\, \delta_{ij}\in \Gamma_{k-1}$, we know that there is a constant $C$ depending only on $n$, $\max_{x\in \mathbb S^n} \lambda_{\max}(x)$, and $\max_{\mathbb S^n} |w|$ such that
\[
\|w\|_{C^2(\mathbb S^n)} \leq C.
\]
Indeed, the second derivative bound follows from the fact that $W_w\in \Gamma_{k-1}\subset \Gamma_1$ and the first derivative bound follows from interpolation.
\par
Now, we apply this fact to $v^m$. Note that, for $v^m$, we have estimates (\ref{unif-c0}) and (\ref{unif-c2}). Therefore,
\[
\|v^m\|_{C^2(\mathbb S^n)} \leq C
\]
for some $C$ independent of $m$. Hence, there exists a subsequence $v^{m_i}$ and a function $v\in C^{1, \gamma}(\mathbb S^n)$ satisfying (\ref{orth}) such that
\[
v^{m_i} \to v \ \textit{ in } \ C^{1, \gamma}(\mathbb S^n) \ \textit{ with }  \ \|v\|_{L^\infty(\mathbb S^n)} =1
\]
Moreover, we also have
\[
\Delta v + n\, v =0 \ \textit{ on } \ \mathbb S^n
\]
in the distribution sense. By linear elliptic theory, $v$ is in fact smooth. Since $v$ satisfies the orthogonal condition (\ref{orth}), we conclude that $v\equiv 0$ on $\mathbb S^n$. But this is a contradiction to the fact that $\|v\|_{L^\infty(\mathbb S^n)} =1$.

\end{proof}

\bigskip

\subsection{Existence via Degree theory} $\\$

\par
The main goal of this subsection is to establish the existence part for equation \eqref{equ-s1}. With the {\it a priori} estimates obtained in the previous section, one may wish to apply the continuity method to get the existence. This leads to study the linearized operator $L$ of the operator $G$ given in \eqref{G-op}. However, as in the discussion for quotient operator in \cite{Guan-Guan}, it is difficult to verify the kernel of $L$. Instead, we will approach the problem using degree theory, which is the only place we need the {\it group invariance} assumption.

\medskip

\par
In order to compute the degree, we need a uniqueness result. The following uniqueness result was proved in \cite{Guan-Ma-Zhou}.

\begin{proposition}\label{uniqueness}
Let $\Gamma$ be an open, symmetric subset in $\mathbb R^n$ and $\Gamma\subset \Gamma_1= \{\,\lambda\, | \,\sum_{j=1}^n \lambda_j>0\}$ is convex. We assume that $Q$ is a $C^{2, \gamma}$ function defined in $\Gamma$ for some $0<\gamma<1$ and satisfies the following conditions in $\Gamma$:
\begin{align}
\label{ellipticity} \frac{\partial Q}{\partial \lambda_i} (\lambda) >0 \ &\text{ for } \forall \ i=1, \cdots, n \text{ and } \lambda \in \Gamma,\\
\label{concavity} Q &\text{ is concave in } \Gamma.
\end{align}
If $u$ is an admissible solution of the equation
\begin{align}\label{equationQ}
Q(u_{ij}+\delta_{ij} u) = Q(I_n) \ \text{ on } \mathbb S^n,
\end{align}
then $u= 1+ \sum_{j=1}^{n+1} a_j x_j$ for some constants $a_1, \cdots, a_{n+1}$.
\end{proposition}

\medskip

\par
From the discussion on the structure of the operator $G$ in Proposition \ref{cone-concavity}, we see that the operator $G$ satisfies conditions \eqref{ellipticity} and \eqref{concavity} with $\Gamma= \Gamma_{k-1}$. As a direct consequence of the above proposition, we obtain
\begin{corollary}\label{uniquenessG0}
Suppose $u$ is an admissible solution of the equation
\begin{align}\label{equationG0}
G_0(u_{ij}+ \delta_{ij} u) = G_0(I_{n}) \ \text{ on } \mathbb S^n,
\end{align}
with $G_0(W) = \frac{\sigma_k}{\sigma_{k-1}}(W)  - \sum_{\ell=0}^{k-2} \frac{\sigma_\ell}{\sigma_{k-1}}(W)$. Then
\[
u= 1+ \sum_{j=1}^{n+1} a_j x_j
\]
 for some constants $a_1, \cdots, a_{n+1}$.
\end{corollary}

\medskip

\medskip

Now, we are ready to prove the main theorem using the derived a priori estimates and the degree method. Here is our main theorem (Theorem \ref{existence proposition}).

\begin{proposition}
Assume $\alpha_\ell(x) \in C^{m, 1}(\mathbb S^n)$ with $m\geq 1$ and $0\leq \ell \leq k-2$ are positive functions and $\alpha(x)\in C^{m, 1}(\mathbb S^n)$. Suppose there is an automorphic group $\mathcal G$ of $\mathbb S^n$ that has no fixed points. If $\alpha(x), \alpha_l(x)$ are invariant under $\mathcal G$, i.e., $\alpha\left(g(x)\right) = \alpha(x)$ and $\alpha_{l}\left(g(x)\right) = \alpha_l(x)$ for all $g\in \mathcal G$ and $x\in \mathbb S^n$, then there exists a $\mathcal G$-invariant admissible solution $u\in C^{m+2, \gamma}, \forall \,0<\gamma<1$, such that $u$ satisfies equation \eqref{equ-s1}.
\par
Moreover, there is a constant $C$ depending only on $\inf_{\mathbb S^n} \alpha_\ell(x)$, $||\alpha_\ell ||_{C^{m, 1}(\mathbb S^n)}$ and $||\alpha||_{C^{m, 1}(\mathbb S^n)}$ such that
\[
|| u ||_{C^{m+2, \gamma}(\mathbb S^n)} \leq C.
\]
In particular, for any $\mathcal G$-invariant $\alpha\in C^{1, 1}(\mathbb S^n)$ and $0< \alpha_\ell(x) \in C^{1, 1}(\mathbb S^n)$ with $0\leq \ell\leq k-2$, equation \eqref{equ-s1} has a $(k-1)$-convex $\mathcal G$-invariant solution.
\end{proposition}

\medskip

\begin{proof}
\par
For $0<\gamma<1, \,\ell\geq 0$ an integer, we set
\[
\mathcal A^{m, \gamma} = \{\  \phi \in C^{m, \gamma}(\mathbb S^n): \phi \text{ is } \mathcal{G}-\text{invariant}\ \}
\]
and
\[
\mathcal O_R= \{ \ w \text{ is (k-1)-convex}, \, w\in \mathcal A^{m, \gamma}: \|w\|_{C^{m, \gamma}(\mathbb S^n)} < R \ \}
\]
for any given large constant $R$.

\medskip

Now, for the fixed $\mathcal G$-invariant functions $\alpha\in C^{m, 1}(\mathbb S^n)$ and $0< \alpha_\ell(x) \in C^{m, 1}(\mathbb S^n)$ with $0\leq \ell\leq k-2$, and for $0\leq t\leq 1$, we consider operator
\begin{eqnarray}\label{deformation}
G_t (u) &=& t\, G(u) + (1-t)\, G_0(u) + t\,\alpha - (1-t) \,G_0(I_n)\\\nonumber
&= &\frac{\sigma_k}{\sigma_{k-1}}(W_u) - \sum_{\ell=0}^{k-2} \left(t\,\alpha_\ell +(1-t)\right) \frac{\sigma_\ell}{\sigma_{k-1}}(W_u) \\\nonumber
&&+ t\,\alpha + (1-t)\,G_0(I_n).
\end{eqnarray}
It is clear that $G_t$ is a nonlinear differential operator that maps $\mathcal A^{m+2, \gamma}$ into $\mathcal A^{m, \gamma}$ for $0<\gamma<1$. Moreover, if $R$ is sufficiently large, $G_t(u)=0$ has no solution on $\partial \mathcal O_R$ by the {\it a priori} estimates established in Proposition \ref{c2-non-deg}. Therefore, the degree of $G_t$ is well-defined (see \cite{Li89}, \cite{Nirenberg}). Using the homotopic invariance of the degree, we have
\[
\deg (G_0, \mathcal O_R, 0) = \deg (G_1, \mathcal O_R, 0).
\]

\par

Recall that any $\mathcal G$-invariant function is orthogonal to $span\{ x_1, \cdots, x_{n+1}\}$ (see Proposition 2.4 in \cite{Guan-Guan}). Therefore, at $t=0$, $u=1$ is the unique solution of \eqref{equationG0} in $\mathcal O_R$ by Corollary \ref{uniquenessG0}. On the other hand, we can compute the degree by the following formula
\begin{align*}
\deg (T_0, \mathcal O_R, 0) = \sum_{\mu_j>0} (-1)^{\beta_j},
\end{align*}
where $\mu_j$ are the eigenvalues of the linearized operator of $G_0$ and $\beta_j$ is its multiplicity.

\medskip

\par
Furthermore, since $G$ is symmetric, it is easy to show that the linearized operator of $G_0$ at $u=1$ is
\[
L_0 = \nu (\Delta + n)
\]
for some constant $\nu>0$. Because the eigenvalues of the Beltrami-Laplace operator $\Delta$ on $\mathbb S^n$ are strictly less than $-n$ except for the first two eigenvalues $0$ and $-n$, there is only one positive eigenvalue of $L_0$ with multiplicity 1, namely $\mu=n\nu$. Therefore,
\[
\deg (G_1, \mathcal O_R, 0) = \deg (G_0, \mathcal O_R, 0) =-1.
\]
It implies that there is an admissible solution of equation \eqref{equ-s1}. The regularity and estimates of the solution $u$ follow directly from Proposition \ref{c2-non-deg} and Schauder estimates.

\end{proof}

\

\subsection{Estimates for the degenerate case}$\\$

\par
In the previous two sub-sections, we derived the a priori estimates and obtained the existence of equation (\ref{equ-s1}) for the non-degenerate case. As it is noted in Proposition \ref{c2-non-deg}, the estimate (\ref{c2-est}) depends on $\inf_{\mathbb S^n} \sum_{\ell=0}^{k-2} \alpha_{\ell}(x)>0$. That is to say, the constants $c_0$ and $C_0$ may go to infinity as this infimum going to zero. Therefore, estimate (\ref{c2-est}) does not work if we consider the degenerate case for equation (\ref{equ-s1}) when only assuming that $\alpha_{\ell}(x)\geq 0$ for $0\leq \ell \leq k-2$. Indeed, we can refine the estimates in the proof of (\ref{c2-est}) and make the a priori $C^2$ estimate independent of $\inf_{\mathbb S^n} \sum_{\ell=0}^{k-2} \alpha_{\ell}(x)$. This allows us to obtain $C^{1, 1}$ solution of equation (\ref{equ-s1}) with non-negative $\alpha_\ell(x)$.

\medskip

\par
First, we consider a simple case
\begin{proposition}\label{c2-deg-sim}
Suppose $u\in C^4(\mathbb S^n)$ is an admissible solution of equation
\begin{eqnarray}\label{14k}
\sigma_k(W_u) + \alpha(x) \sigma_{k-1}(W_u) = f(x) \ \textit{ on } \mathbb S^n
\end{eqnarray}
with $\alpha(x) \in C^{1, 1}(\mathbb S^n)$ and $f(x) = \left( g(x) \right)^p$ for some $g(x) \in C^{1, 1}(\mathbb S^n)$ and $p \geq k-1$. Then, there exists a positive constant $C$ depending only on $n$, $\|\alpha(x)\|_{C^{1, 1}(S^n)}$ and $\|g(x)\|_{C^{1,1}(S^n)}$ such that
\begin{eqnarray}\label{c2-est-deg}
\left|\lambda_i (x) \right|\leq C \ \textit{ on } \mathbb S^n
\end{eqnarray}
where $\lambda_i(x)$ are the eigenvalues of $W_u(x) = u_{ij} + u\, \delta_{ij}$.
\end{proposition}

\medskip

\par
\begin{proof} We note that the constants in the first inequality of (\ref{7k}) do not depend on $\inf_{\mathbb S^n}\alpha_\ell(x)$ for $0\leq \ell\leq k-2$. Thus, if we restrict to the case that $\alpha_\ell(x) \equiv 0$ for $1\leq \ell\leq k-2$ and $\alpha_0(x) = f(x)$ with $f$ as given above. We obtain
\begin{eqnarray}\label{15k}
 0\geq  G^{ij} H_{ij} = - \left( nk f - \Delta f + \frac{|\nabla f|^2}{f}\right) \frac{1}{\sigma_{k-1}} + H\sum_{i=1}^n G^{ii} - C(\Delta \alpha, \alpha)
\end{eqnarray}
By the assumption on $f(x)= \left( g(x) \right)^p$, we have
\[
\Delta f = \Delta \left( g(x) \right)^p = p \left( g(x) \right)^{p-1} \Delta g + p(p-1) \left( g(x) \right)^{p-2} | \nabla g|^2 \leq C \left( g(x) \right)^{p-1}.
\]
for some constant $C>0$ depend on $\|g\|_{C^{1, 1}(\mathbb S^n)}$ and independent of $\inf_{\mathbb S^n} g(x)$. Here, to obtain the inequality above, we used the fact that
\[
|\nabla g|^2 \leq C \| g\|_{C^{1, 1}}\cdot g
\]
 for non-negative $C^{1, 1}$ function $g(x)$. Moreover, we also have
\[
\frac{|\nabla f|^2}{f} = \frac{ \left( g(x) \right)^{2(p-2)}|\nabla g|^2}{\left( g(x) \right)^p} \leq C \left( g(x) \right)^{p-1}.
\]
Therefore, we obtain
\[
- \left( nk f - \Delta f + \frac{|\nabla f|^2}{f}\right) \geq -C_0  \left( g(x) \right)^{p-1}
\]
for some constants $C_0>0$ independent of $\inf_{\mathbb S^n} f(x)$.

\medskip

\par
On the other hand, by a similar computation as in previous section,
\[
\sum_{i=1}^n G^{ii} = \frac{n-k+1}{k} + (n-k+2)f \, \frac{\sigma_{k-2}}{\sigma_{k-1}^2} .
\]
Then, (\ref{15k}) implies
\begin{eqnarray}\label{16k}
0&\geq & \frac{n-k+1}{k}  H + C(n, k) \left( g(x) \right)^p \frac{ \sigma_1 \sigma_{k-2}}{\sigma_{k-1}^2} - C_0 \left( g(x) \right)^{p-1} \frac{1}{\sigma_{k-1}} -C
\end{eqnarray}
since $H$ is large. In order to get the the upper bound for $H$, we need to show that $\frac{n-k+1}{2k}H$ is the higher order term in the above inequality.

\par
Indeed, if $C_0\, g^{p-1} \frac{1}{\sigma_{k-1}} \leq \delta H$ for some small $\delta$, say $\delta =\frac{n-k+1}{4k}$, then we are done since this gives
\[
\frac{n-k+1}{4k} H - C_0 \, g^{p-1} \frac{1}{\sigma_{k-1}} \geq 0.
\]
And it follows from (\ref{16k}) that $H\leq  C_1$.
\medskip

\par
In the following, we will assume that $C_0\, g^{p-1} \frac{1}{\sigma_{k-1}} \geq \delta H$ for some $\delta>0$, i.e., assume that
\begin{eqnarray}\label{17k}
 \frac{1}{\sigma_{k-1}}  \geq C_2 g^{1-p} H,
\end{eqnarray}
for some constant $C_2>0$ independent of $\inf_{\mathbb S^n} g(x)$.
\par
Again, by Newton-Maclaurin inequality, it follows that
\[
\left( \frac{\sigma_{k-2}}{\sigma_{k-1}}\right)^{k-3} \geq \frac{\sigma_{k-3}}{\sigma_{k-2}} \cdots \frac{\sigma_{1}}{\sigma_{2}} = \frac{\sigma_{1}}{\sigma_{k-2}}
\]
and then
\begin{eqnarray}\label{18k}
\sigma_{k-2} \geq \sigma_{1}^{\frac{1}{k-2}}\sigma_{k-1}^{\frac{k-3}{k-2}}.
\end{eqnarray}
From this estimate and (\ref{17k}), we obtain
\begin{eqnarray}\label{19k}
\nonumber C(n, k)g^p \, \frac{\sigma_1 \,\sigma_{k-2}}{\sigma_{k-1}^2} - C_0 \,g^{p-1} \frac{1}{\sigma_{k-1}} &=& C(n, k)\, g^p \frac{1}{\sigma_{k-1}}\left( \frac{\sigma_1 \sigma_{k-2}}{\sigma_{k-1}}- \tilde C_1 g^{-1}\right)\\\nonumber
&\geq & C(n, k) g^p \frac{1}{\sigma_{k-1}}\left( \sigma_1 \sigma_1^{\frac{1}{k-2}}\sigma_{k-1}^{-\frac{1}{k-2}}- \tilde C_1 g^{-1}\right)\\\nonumber
&\geq & C(n, k) g^p \frac{1}{\sigma_{k-1}}\left( \sigma_1^{1+\frac{1}{k-2}}\left(C_2 g^{1-p}\sigma_1\right)^{\frac{1}{k-2}} - \tilde C_1 g^{-1}\right)\\
&\geq & C(n, k) g^p \frac{1}{\sigma_{k-1}}\left(\tilde C_2 \sigma_1^{1+\frac{2}{k-2}}  g^{\frac{1-p}{k-2}} - \tilde C_1 g^{-1}\right)
\end{eqnarray}
where $\tilde C_1$ and $\tilde C_2$ are constants depending on $C_0, C_2$ and $C(n, k)$.

\par
If $p\geq k-1$, we have $\frac{1-p}{k-2}\leq -1$, and thus the right hand side of (\ref{19k}) is positive. Plugging this back to (\ref{16k}), we get the desired estimate $H\leq C$. This proves the upper bound in (\ref{c2-est-deg}). The lower bound follows from the admissible condition.

\end{proof}

\bigskip

In the rest of this section, we extend the trick to general case and prove the $C^{1,1}$ estimate for the degenerate case of equation (\ref{equ-s1}).

\begin{proposition}\label{c2-deg-com}
Suppose $u\in C^4(\mathbb S^n)$ is an admissible solution of equation
\begin{eqnarray}\label{20k}
\sigma_k(W_u) + \alpha(x) \sigma_{k-1}(W_u) = \sum_{\ell=0}^{k-2} \alpha_\ell(x) \, \sigma_\ell (W_u) \ \textit{ on } \mathbb S^n
\end{eqnarray}
with $\alpha(x) \in C^{1, 1}(\mathbb S^n)$ and with $\alpha_\ell(x) = \left(g_\ell(x)\right)^{p_\ell}$ for some $0\leq g_\ell(x) \in C^{1, 1}(\mathbb S^n)$ and $p_\ell \geq k-\ell$. Then, there exists a positive constant $C$ depending on $n$, $\|\alpha(x)\|_{C^{1, 1}(S^n)}$ and $\|g_\ell(x)\|_{C^{1,1}(S^n)}$ such that
\begin{eqnarray}\label{c2-est-deg}
\left|\lambda_i (x) \right|\leq C \ \textit{ on } \mathbb S^n
\end{eqnarray}
where $\lambda_i(x)$ are the eigenvalues of $W_u(x) = u_{ij} + u\, \delta_{ij}$.
\end{proposition}

\begin{proof}
Similar as the proof of  (\ref{7k}) (also (\ref{15k})), we get
\begin{eqnarray}\label{21k}
0&\geq& G^{ij}H_{ij}\\\nonumber
&\geq & H\sum_i G^{ii} - \sum_{\ell=0}^{k-2} \left( n(k-\ell) \alpha_\ell - \Delta \alpha_\ell + \frac{|\nabla \alpha_\ell |^2 }{\alpha_\ell} \right) \frac{\sigma_\ell}{\sigma_{k-1}}-C(\Delta \alpha, \alpha)
\end{eqnarray}
First, direct computation implies
\begin{eqnarray*}\label{sumG}
\sum_{i=1}^n G^{ii} &=& \sum_{i=1}^n\left\{ \left( \frac{\sigma_k}{\sigma_{k-1}}\right)^{ii} - \sum_{\ell=0}^{k-2} \alpha_\ell \left( \frac{\sigma_\ell}{\sigma_{k-1}}\right)^{ii} \right\}\\
&=& \frac{n-k+1}{k} +  \sum_{\ell=0}^{k-2}\alpha_\ell \left\{ (n-k+2) \frac{\sigma_\ell\sigma_{k-2}}{\sigma_{k-1}^2} - (n-\ell+1) \frac{\sigma_{\ell-1}\sigma_{k-1}}{\sigma_{k-1}^2}\right\}
\end{eqnarray*}
Again, by standard Newton-Maclaurin inequality, for $1\leq \ell<k-1$,
\[
\frac{\sigma_{k-2}}{\sigma_{k-1}}\geq \frac{k-1}{\ell}\frac{n-\ell+1}{n-k+2} \frac{\sigma_{\ell-1}}{\sigma_{\ell}},
\]
from which follows that
\begin{eqnarray*}
\sum_{i=1}^n G^{ii} &\geq &  \frac{n-k+1}{k} +  \sum_{\ell=0}^{k-2}C(n, k, \ell)\, \alpha_\ell\, \frac{\sigma_\ell\sigma_{k-2}}{\sigma_{k-1}^2}.
\end{eqnarray*}
Then,
\begin{eqnarray}\label{22k}
0\geq G^{ij}H_{ij} &\geq&  \frac{n-k+1}{k} H + \sum_{\ell=0}^{k-2}C(n, k, \ell) \alpha_\ell\,\frac{\sigma_1 \sigma_\ell\sigma_{k-2}}{\sigma_{k-1}^2}\\\nonumber
&&- \sum_{\ell=0}^{k-2} \left( n(k-\ell) \alpha_\ell - \Delta \alpha_\ell + \frac{|\nabla \alpha_\ell |^2 }{\alpha_\ell} \right) \frac{\sigma_\ell}{\sigma_{k-1}}-C(\Delta \alpha, \alpha).
\end{eqnarray}
Recall that $\alpha_\ell(x) =  \left(g_\ell(x)\right)^{p_\ell}$ for some $0\leq g_\ell(x) \in C^{1, 1}$ with $\ell=0, 1, \cdots, k-2$. Using the same trick as the previous simple case, we can get
\[
- \left( n(k-\ell) \alpha_\ell - \Delta \alpha_\ell + \frac{|\nabla (\alpha_\ell)|^2}{\alpha_\ell}\right) \geq -C_\ell  \left( g_\ell(x) \right)^{p_\ell-1},
\]
where $C_\ell>0$ is a constant independent of $\inf_{\mathbb S^n} g_\ell(x)$ and $\ell=0, 1, \cdots, k-2$. Put this back to (\ref{22k}), we arrive
\begin{eqnarray}\label{23k}
0\geq G^{ij}H_{ij}
&\geq &  \frac{n-k+1}{k} H + \sum_{\ell=0}^{k-2}C(n, k, \ell) \Big(g_\ell(x)\Big)^{p_\ell} \frac{\sigma_1 \sigma_\ell\sigma_{k-2}}{\sigma_{k-1}^2}\\\nonumber
&&- \sum_{\ell=0}^{k-2} \tilde C_\ell  \Big(g_\ell(x)\Big)^{p_\ell-1} \frac{\sigma_\ell}{\sigma_{k-1}} - C(\Delta \alpha, \alpha).
\end{eqnarray}
Again, to get an upper bound for $H$, we need to show that $\frac{n-k+1}{k} H $ is the dominated term. For this purpose, we divide the indices $\ell=0, 1, \cdots, k-2$ into two groups:
\[
\mathcal N = \left\{ \ \ell \ \bigg| \ \tilde C_\ell\,  g_\ell^{p_\ell-1} \frac{\sigma_\ell}{\sigma_{k-1}} \leq \delta H, \ \text{ for some small constant } \delta, \text{ say }  \frac{n-k+1}{4k} \right\};
\]
and
\[
\mathcal N^c = \left\{ \ \ell \ \bigg|\  \tilde C_\ell\,  g_\ell^{p_\ell-1} \frac{\sigma_\ell}{\sigma_{k-1}} \geq \delta H, \ \text{ for some small constant } \delta \right\}.
\]
It is easy to see that, for those $\ell\in \mathcal N$, we are fine since, for $\ell\in \mathcal N$, the trouble term  $\tilde C_\ell \, g_\ell^{p_\ell-1} \frac{\sigma_\ell}{\sigma_{k-1}} $ could be dominated by $\frac{n-k+1}{4k} H$.
\medskip
\par
Now, we deal with the terms with $\ell\in \mathcal N^c$. By the definition of $\ell \in \mathcal N^c$, we have
\[
\frac{\sigma_\ell}{\sigma_{k-1}} \geq \delta \tilde C_\ell^{-1} \sigma_1 g_\ell ^{1-p_\ell}.
\]
Using the same trick as (\ref{18k}), we obtain
\[
\frac{\sigma_{k-2}}{\sigma_{k-1}} \geq \left( \frac{\sigma_{\ell}}{\sigma_{k-1}}\right)^{\frac{1}{k-1-\ell}} \geq \bar C_\ell\, g_\ell ^{\frac{1-p_\ell}{k-1-\ell}} \sigma_1^{\frac{1}{k-1-\ell}}.
\]
Thus, for $\ell\in \mathcal N^c$,
\begin{eqnarray*}
C(n, k, \ell)\, g_\ell^{p_\ell}\, \frac{\sigma_1 \sigma_\ell \sigma_{k-2}}{\sigma_{k-1}^2} - \tilde C_\ell \,g_\ell^{p_\ell-1}\frac{\sigma_\ell}{\sigma_{k-1}}& =& C(n, k, \ell) \,g_\ell^{p_\ell}\, \frac{\sigma_\ell}{\sigma_{k-1}} \left( \sigma_1 \frac{\sigma_{k-2}}{\sigma_{k-1}} - \hat C_\ell g_\ell^{-1} \right)\\
&\geq & C(n, k, \ell) g_\ell^{p_\ell} \frac{\sigma_\ell}{\sigma_{k-1}} \left( \bar C_\ell \sigma_1^{1+\frac{1}{k-1-\ell}} g_\ell^{\frac{1-p_\ell}{k-1-\ell}} -  \hat C_\ell g_\ell^{-1} \right)
\end{eqnarray*}

\par
If $p_\ell\geq k-\ell$, we have $\frac{1-p_\ell}{k-1-\ell}\leq -1$ and then the last term in the above inequalities is non-negative.
\par
Finally, by (\ref{23k}), we get
\begin{eqnarray*}
0\geq G^{ij}H_{ij}&\geq&  \frac{n-k+1}{k} H + \sum_{\ell\in \mathcal N}\left( C(n, k, \ell) g_\ell^{p_\ell} \frac{\sigma_1 \sigma_\ell\sigma_{k-2}}{\sigma_{k-1}^2}- \tilde C_\ell  g_\ell^{p_\ell-1} \frac{\sigma_\ell}{\sigma_{k-1}}\right) \\
&&+  \sum_{\ell\in \mathcal N^c}\left( C(n, k, \ell) g_\ell^{p_\ell} \frac{\sigma_1 \sigma_\ell\sigma_{k-2}}{\sigma_{k-1}^2}- \tilde C_\ell  g_\ell^{p_\ell-1} \frac{\sigma_\ell}{\sigma_{k-1}}\right) - C(\Delta \alpha, \alpha)\\
&\geq& \frac{n-k+1}{2k} H - C(\Delta \alpha, \alpha)
\end{eqnarray*}
which gives the upper bound $H\leq C$.

\end{proof}

\medskip

As mentioned in the introduction, an interesting PDE problem left open is to find the optimal power $p_{\ell}$ in the requirement of $\|\alpha_\ell^{1/ p_\ell}\|\in C^{1, 1}$ to obtain the above $C^{1, 1}$ estimate for the degenerate fully nonlinear equations. The estimates obtained in Propositions \ref{c2-deg-sim} and \ref{c2-deg-com} can be viewed as studying those problems on compact manifolds case. In general, this would be simpler than the domain case since one do not need to handle the $C^{1, 1}$ estimates on the boundary of the domain.

\par
In the simple case, Proposition \ref{c2-deg-sim} derived the estimates under the condition that $f^{1/(k-1)}\in C^{1, 1}(\mathbb S^n)$. In viewing of the examples given in \cite{DPZ} on compact manifolds, the power $1/(k-1)$ is optimal.
\par
In Proposition \ref{c2-deg-com}, we assume that $p_\ell \geq k-\ell$ for $\ell=0, 1, \cdots, k-2$. Indeed, by using the same trick as the previous simple case, we can weak the condition for $p_0$ to be $p_0 \geq k-1$. But it is still an interesting question that whether the estimate still holds if one assumes $p_\ell \geq k-\ell-1$ for all $\ell =1, \cdots, k-2$.

\

\section{Dirichlet Problem}
In this section, we consider the following Dirichlet problem
\begin{equation}\label{DP}
\begin{cases}
\sigma_k(D^2 u) + \alpha(x) \sigma_{k-1}(D^2 u) = \sum_{\ell=0}^{k-2} \alpha_\ell(x) \sigma_\ell(D^2 u) & \text{ in } \Omega\\
u = \phi & \text{ on } \partial \Omega.
\end{cases}
\end{equation}
where $\Omega\in \mathbb R^n$ is a bounded domain with smooth boundary $\partial\Omega$, $\phi \in C^{\infty}(\partial\Omega)$ and $D^2 u$ is the Hessian matrix of the function $u$.

\medskip

\subsection{Interior $C^{1, 1}$ estimate}$\\$

\noindent
We will derive the interior $C^{1, 1}$ estimates for the solution of the {\it degenerate} equation (\ref{DP}) under the following assumptions
\begin{itemize}
\item $\alpha_\ell(x) = \left(g_\ell(x)\right)^{p_\ell}$ for some $0\leq g_\ell(x) \in C^{1, 1}$ and $p_\ell \geq k-\ell$;
\item  $\alpha(x) \in C^{1, 1}$.
\end{itemize}

 \begin{proposition}
 Let $u\in C^4(\Omega) \cap C^2(\bar\Omega)$ be an admissible solution, that is $D^2u\in \Gamma_{k-1}$, of equation (\ref{DP}) in $\Omega$. Then
 \begin{eqnarray}\label{c2-Dir-int}
  \sup_{\Omega} | D^2 u| \leq C \big(1+ \sup_{\partial\Omega} |D^2 u|\big),
   \end{eqnarray}
   where $C$ is a constant depending on $\|g_\ell\|_{C^{1, 1}(\Omega)}$ and $\|\alpha(x) \|_{C^{1, 1}(\Omega)}$, but independent of $\inf_{\Omega} g_\ell(x)$ with $0\leq \ell\leq k-2$.
 \end{proposition}

\begin{proof}
Consider $\omega(x) = \Delta u(x) + \frac{1}{2}M|x|^2$ for some large constant $M$ to be determined later. Suppose that $\omega$ attains its maximum at $x_0$. If $x_0\in \partial\Omega$, we are done. So to prove the estimate, we assume $x_0\in \Omega$. By rotating the coordinates, we may suppose that $D^2 u$ is diagonal at $x_0$. Then, at $x_0$, we have
\begin{eqnarray*}
0= \omega_i(x_0) = \sum_{k=1}^n u_{kki} + Mx_i;
 \end{eqnarray*}
and
\begin{eqnarray*}
0 \geq \omega_{ij}(x_0) = \sum_{k=1}^n u_{kkij} +M\, \delta_{ij}.
\end{eqnarray*}
This implies, at $x_0$,
\begin{eqnarray*}
0 \geq G^{ij} \omega_{ij} = G^{ij} \Delta(u_{ij}) + M \sum_{i, j} G^{ij}
\end{eqnarray*}
By differentiate equation (\ref{DP}) twice, we can compute
\begin{eqnarray*}
-\Delta \alpha &=& \sum_{p=1}^nG_k^{ij, rs} u_{ijp}u_{rsp} + G_k^{ij}  \Delta u_{ij}
 +\sum_{p=1}^n\sum_{\ell=0}^{k-2} \alpha_\ell G_{\ell}^{ij, rs} u_{ijp}u_{rsp}+ \sum_{\ell=0}^{k-2} \alpha_\ell G_\ell^{ij}\Delta u_{ij} \\
 &&+ 2 \sum_{p=1}^n\sum_{\ell=0}^{k-2} (\alpha_\ell)_p G_\ell^{ij} u_{ijp} + \sum_{\ell=0}^{k-2} \Delta(\alpha_\ell) G_\ell
\end{eqnarray*}
It follows that
\begin{eqnarray*}
G^{ij} \Delta u_{ij} &=& - \sum_{p=1}^nG_k^{ij, rs} u_{ijp}u_{rsp}
- \sum_{p=1}^n\sum_{\ell=0}^{k-2} \alpha_\ell G_\ell^{ij, rs} u_{ijp}u_{rsp} - 2\sum_{p=1}^n\sum_{\ell=0}^{k-2} (\alpha_\ell)_p G_\ell^{ij} u_{ijp} \\\nonumber
&&- \sum_{\ell=0}^{k-2} \Delta(\alpha_\ell) G_\ell -\Delta \alpha.
\end{eqnarray*}
Then, following the same trick as estimating (\ref{est-concave2}), we can make use of the concavity of operator $G_k$ and $\left(- \frac{1}{G_l}\right)^{\frac{1}{k-1-l}}$ for $l=1, \cdots, k-2$ and obtain
\begin{eqnarray}\label{estimate1}
0&\geq& G^{ij}\omega_{ij}  \geq  \sum_{\ell=0}^{k-2} \left(- \Delta\alpha_\ell + \frac{|\nabla\alpha_\ell |^2}{\alpha_\ell}\right) G_\ell + M\sum_{i=1}^n G^{ii} -\Delta \alpha.
\end{eqnarray}
On the other hand, we can compute
\begin{eqnarray*}
\sum_{i=1}^n G^{ii} &\geq &  \frac{n-k+1}{k} +  \sum_{\ell=0}^{k-2}C(n, k, \ell) \alpha_\ell \frac{\sigma_\ell\sigma_{k-2}}{\sigma_{k-1}^2}.
\end{eqnarray*}

\medskip

Now, using the assumption that $\alpha_\ell(x) =  \left(g_\ell(x)\right)^{p_\ell}$ for some $0\leq g_\ell(x) \in C^{1, 1}$ for $ 0\leq \ell \leq k-2$, we have
\[
 \Delta \alpha_\ell - \frac{|\nabla \alpha_\ell|^2}{\alpha_\ell}  \geq -C_\ell  \left( g_\ell(x) \right)^{p_\ell-1},
\]
where $C_\ell>0$ is a constant independent of $\inf g_\ell(x)$. Putting this back to (\ref{estimate1}), we arrive
\begin{eqnarray}\label{estimate2}
\nonumber 0&\geq& G^{ij}\omega_{ij}\\
&\geq &  \frac{n-k+1}{k} M + M\sum_{\ell=0}^{k-2}C(n, k, \ell) \Big(g_\ell(x)\Big)^{p_\ell}\, \frac{ \sigma_\ell\, \sigma_{k-2}}{\sigma_{k-1}^2} - \sum_{\ell=0}^{k-2} \tilde C_\ell  \Big(g_\ell(x)\Big)^{p_\ell-1} \frac{\sigma_\ell}{\sigma_{k-1}} - \Delta \alpha.
\end{eqnarray}
We note that, if there is a constant $N$ such that
\[
\tilde C_\ell \, g_\ell^{p_\ell-1} \frac{\sigma_\ell}{\sigma_{k-1}} \leq N, \ \text{ for } \forall \, \ell=0, \cdots, k-2,
\]
then we obtain the desired estimate since inequality (\ref{estimate2}) becomes
\begin{eqnarray*}
0 \geq  \frac{n-k+1}{k} M - (k-1)N - \Delta \alpha.
\end{eqnarray*}
If we take $M$ large enough, this gives contradiction. Therefore, the maximum point of $\omega(x)$ can not be in the interior of $\Omega$ and we get the estimate (\ref{c2-Dir-int}). Based on this observation, we divide the indices into two groups, for some constant $N$,
\[
\mathcal N = \{ \ \ell \ | \ \tilde C_\ell  g_\ell^{p_\ell-1} \frac{\sigma_\ell}{\sigma_{k-1}} \leq N\}, \ \  \textit{ and } \ \  \mathcal N^c = \{\  \ell\  | \ \tilde C_\ell g_\ell^{p_\ell-1} \frac{\sigma_\ell}{\sigma_{k-1}}> N\}
\]
\par
For those $\ell\in \mathcal N$, it is easy to handle since
\[
\tilde C_\ell  \Big(g_\ell(x)\Big)^{p_\ell-1} \frac{\sigma_\ell}{\sigma_{k-1}}\leq \frac{n-k+1}{k} M \]
if we take $M$ large enough.

\par
Next, we deal with terms that $\ell\in \mathcal N^c$. By the definition of $\mathcal N^c$, we have
\[
\frac{\sigma_\ell}{\sigma_{k-1}} > N\tilde C_\ell^{-1}  g_\ell^{1-p_\ell}
\]
Using the same trick as (\ref{18k}), we obtain
\[
\frac{\sigma_{k-2}}{\sigma_{k-1}} \geq \left( \frac{\sigma_{\ell}}{\sigma_{k-1}}\right)^{\frac{1}{k-1-\ell}} \geq \bar C_\ell g_\ell ^{\frac{1-p_\ell}{k-1-\ell}}.
\]
Then, for $\ell\in \mathcal N^c$,
\begin{eqnarray*}
 C(n, k, \ell)\,M g_\ell^{p_\ell} \frac{ \sigma_\ell \sigma_{k-2}}{\sigma_{k-1}^2} - \tilde C_\ell g_\ell^{p_\ell-1}\frac{\sigma_\ell}{\sigma_{k-1}}&=&  \tilde C_\ell g_\ell^{p_\ell} \frac{\sigma_\ell}{\sigma_{k-1}} \left(\hat C_\ell M \frac{\sigma_{k-2}}{\sigma_{k-1}} - g_\ell^{-1} \right)\\
&\geq & \tilde C_\ell g_\ell^{p_\ell} \frac{\sigma_\ell}{\sigma_{k-1}} \left( \bar C_\ell \hat C_\ell M  g_\ell^{\frac{1-p_\ell}{k-1-\ell}} -   g_\ell^{-1} \right)
\end{eqnarray*}
If $p_\ell\geq k-\ell$, we have $\frac{1-p_\ell}{k-1-\ell}\leq -1$ and then the last term in the above inequalities is non-negative if we take $M$ large enough. This again gives contradiction in (\ref{estimate2}).

\end{proof}

\medskip

\subsection{Boundary $C^2$ estimate}$\\$

\par
We follow the idea in \cite{GuanB1, GuanB2} to derive the boundary $C^2$ estimate under the assumption on sub-solutions.

\begin{theorem}
Let $u\in C^4(\Omega) \cap C^2(\bar\Omega)$ be an admissible solution of the Dirichlet problem (\ref{DP}) with $\alpha(x)\in C^{1, 1}(\bar\Omega)$ and $\alpha_\ell(x)>0\in C^{1, 1}(\bar\Omega)$.
Assume that there exists an admissible subsolution $\underline u \in C^2(\bar \Omega)$, that is
\begin{equation}\label{subsolution}
\begin{cases}
\sigma_k(D^2 \underline u) + \alpha(x) \sigma_{k-1}(D^2 \underline u) \geq \sum_{\ell=0}^{k-2} \alpha_\ell(x) \sigma_\ell(D^2 \underline u) & \text{ in } \Omega\\
\underline u = \phi & \text{ on } \partial \Omega.
\end{cases}
\end{equation}
Then, there exists a constant $C$ depending on $\|u\|_{C^1(\bar\Omega)}$, $\|\underline u\|_{C^2(\bar\Omega)}$, $\|\phi\|_{C^3}$ and $\inf_{\partial\Omega} \alpha_\ell(x)$ with $0\leq \ell \leq k-2$ such that
\begin{eqnarray}\label{boundaryestimate}
\max_{\partial\Omega} |D^2 u|\leq C.
\end{eqnarray}
\end{theorem}

\medskip

\par
We assume that $\phi\in C^{\infty}(\partial\Omega)$ is extended smoothly to $\bar \Omega$ and still denoted $\phi$. Before proceeding the estimate, we recall that our equation can be written as
\begin{eqnarray}\label{equ-G-form}
G(\lambda(D^2 u)): = G_k(\lambda(D^2 u)) + \sum_{\ell=0}^{k-2} \alpha_{\ell}(x) \, G_{\ell}(\lambda(D^2 u))= - \alpha(x),
\end{eqnarray}
where $G_k$ and $G_\ell$ are defined in (\ref{operatorG}). The subsolution condition given in (\ref{subsolution}) is equivalent to
\begin{equation}\label{subsolution1}
\begin{cases}
G(D^2 \underline u) \geq - \alpha(x) & \text{ in } \Omega\\
\underline u = \phi & \text{ on } \partial \Omega.
\end{cases}
\end{equation}

\medskip
We will follow the main idea in \cite{GuanB1, GuanB2} to establish the boundary estimates. Before starting the proof, we want to remark that the boundary estimates in \cite{GuanB1, GuanB2} deal with a general family of elliptic equations $F(D^2 u) = f(\lambda(D^2 u))$ with $f$ satisfying ellipticity and concavity properties. Here, we are working with operator $G(\lambda(D^2 u))$ defined in (\ref{equ-G-form}). As discussed in \S 2, the operator $G$ has desired good algebraic properties: ellipticity and concavity. Therefore, the lemmas in the proof only relying on the algebraic properties follow the way same as \cite{GuanB1, GuanB2}. However, it is not hard to see that our operator $G$ also depends on the variable $x$ due to the non-constants coefficient functions $\alpha_{\ell}(x)$. Because of this new feature, new terms will come out when differentiating the equations and we need to modify the arguments and barrier functions to overcome the new trouble terms.

\medskip

To estimate the second derivatives of $u$ at an arbitrary point on $\partial\Omega$, we may assume that the point is the origin of $\mathbb R^n$ and that the positive $x_n$ axis is in direction of the interior normal to $\partial\Omega$ at $0$. Near the origin $0$, $\partial\Omega$ can be represented as a graph
\begin{eqnarray}\label{bgraph}
x_n = \rho(x') = \frac{1}{2} \sum_{\gamma, \beta < n} B_{\gamma\beta} x_{\gamma}x_{\beta} + O\left( |x'|^3 \right), \ x'= (x_1, \cdots, x_{n-1} ).
\end{eqnarray}
\par
First, notice that $(u- \underline u) (x', \rho(x'))=0$ for any boundary point and it follows that
\begin{eqnarray}\label{dtangent}
(u- \underline u)_{\gamma\beta}(0) = - (u-\underline u)_{n}(0) B_{\gamma\beta} ,  \  \gamma, \beta<n
\end{eqnarray}
Hence, we obtain the second derivative estimate along the tangential directions:
\begin{eqnarray}\label{dtangent1}
\left|u_{\gamma\beta}(0) \right| \leq C, \ \gamma, \beta<n.
\end{eqnarray}

\medskip
\par
Next, we proceed to estimate $u_{\gamma n}(0)$ for $\gamma<n$. We follow the idea from \cite{GuanB1} and employ a barrier function of the following form
\begin{eqnarray}\label{barrierv}
v= (u- \underline u) + t d - \frac{N}{2}d^2,
\end{eqnarray}
where $d$ is the distance function from $\partial\Omega$ and $t, N$ are positive constants to be determined. We take $\delta>0$ small enough such that $d$ is smooth in
$\Omega_{\delta}:=\Omega \cap B_{\delta}(0)$.
The key ingredient is the following lemma which follows from Lemma 2.1 in \cite{GuanB1}.

\begin{lemma}\label{keylemma}
There exist some uniform positive constants $t, \delta, \epsilon$ sufficiently small and $N$ sufficiently large such that $v$ satisfies the following
\begin{equation}
\begin{cases}
G^{ij} v_{ij} \leq -\epsilon \left(1+ \sum_{i=1}^n G^{ii}\right) & \text{ in } \Omega_{\delta}\\
v\geq 0 & \text{ on } \bar \Omega_{\delta}.
\end{cases}
\end{equation}
\end{lemma}
\begin{proof}
The proof is essentially the same as Lemma 2.1 in \cite{GuanB1} since it only makes use of the algebraic properties of the operator. For completeness, we include an argument here.

\medskip

We note that to ensure $v\geq 0$ in $\bar\Omega_{\delta}$, we may require $\delta \leq \frac{2t}{N}$ for $t, N$ to be fixed later. By direct computation, we have
\begin{eqnarray}
G^{ij} v_{ij} &=& G^{ij} (u- \underline u)_{ij} + (t-Nd)G^{ij} d_{ij} - N G^{ij} d_id_j \nonumber\\
&\leq& C(t+Nd) \sum_k G^{ii} + \sum_{i, j} G^{ij} (u- \underline u)_{ij}- N G^{ij} d_id_j
\end{eqnarray}
Let $\lambda=\lambda(u_{ij})$ be the eigenvalues of $u_{ij}$ and $\mu=\mu(\underline u_{ij})$ be the eigenvalues of $\underline u_{ij}$. Denote $\nu_{\chi}:= \frac{D G(\chi)}{\left | D G(\chi)\right|}$ to be the unit normal vector to the level hypersurface $\partial\Gamma^{G(\chi)}$ for $\chi\in \Gamma_{k-1}$.
Note that $\{\mu(x)\ : \ x\in \bar \Omega\}$ is a compact subset of $\Gamma_{k-1}$. There exists a uniform constant $\beta \in (0, \frac{1}{2\sqrt{n}})$ such that
\[
\nu_{\mu(x)} - 2\beta {\bf1} \in \Gamma_n, \ \forall x\in \bar\Omega.
\]

\par
We consider two cases:

\par
Case {\bf (a)}: $|\nu_{\lambda} - \nu_{\mu}| \geq \beta$. It follows from Lemma 2.1 in \cite{GuanB1} that
\[
\sum G^{ii} (\lambda) (\mu_i - \lambda_i) \geq \epsilon (1+ \sum G^{ii} (\lambda) )
\]
for some uniform constant $\epsilon>0$. This estimate makes use of the concavity of operator $G(D^2 u)$.
Then, we have
\[
G^{ij} v_{ij} \leq C(t+Nd) \sum G^{ii} - \epsilon (1+ \sum G^{ii}) - N \sum G^{ij} d_i d_j.
\]
Taking $t, \delta$ small enough, such that $C(t+Nd) < \frac{\epsilon}{2}$, we arrive
\[
G^{ij} v_{ij} \leq - \frac{\epsilon}{2} (1+ \sum G^{ii}).
\]
\smallskip

\par
Case {\bf (b)}: $|\nu_{\lambda} - \nu_{\mu}| < \beta$. This implies that
\[
G^{ii} \geq \frac{\beta}{\sqrt{n}} \sum_j G^{jj}, \ \forall \,1\leq i\leq n.
\]
On the other hand, by the concavity of $G$, we have
\[
G^{ij}(D^2 u) ( u_{i j} - \underline u_{ij}) \leq G(D^2 u) - G(D^2 \underline u) \leq 0.
\]
Thus, we have
\begin{eqnarray*}
G^{ij} v_{ij} &\leq& C(t+Nd) \sum G^{ii} - N \sum G^{ij} d_i d_j\\
&\leq &  C(t+Nd) \sum G^{ii} - \frac{N\beta}{\sqrt{n}}\sum_i G^{ii} \sum_{i=1}^n |d_i|^2\\
&=& C(t+Nd) \sum G^{ii} - \frac{N\beta}{\sqrt{n}}\sum_i G^{ii},
\end{eqnarray*}
since $\sum_{i=1}^n |d_i|^2= |\nabla d|^2=1$. So, if we take $t, \delta$ small enough, we obtain
\[
G^{ij} v_{ij}  \leq - C\sum_i G^{ii}\leq - \epsilon (1+ \sum_i G^{ii})
\]
if $|\lambda|>R$ for some large constant $R$.

\par
If we are in the case that $|\lambda |<R$, it follows that $c_1I_n \leq G^{ij} < C_1I_n$ for some constants $C_1\geq c_1\geq 0$. Then, $\sum G^{ij} d_i d_j \geq c_1$ and we have
\begin{eqnarray*}
G^{ij} v_{ij} &\leq& C(t+Nd) \sum G^{ii} - N \sum G^{ij} d_i d_j
\end{eqnarray*}
Taking $t, \delta$ small enough, we get
\[
G^{ij} v_{ij}  \leq - C \leq - \epsilon (1+ \sum_i G^{ii}).
\]
\end{proof}

\medskip

\par
Now, we are in the place to give the estimate for $|u_{\gamma n}(0)|$ for $\gamma<n$. Recall the local representation of boundary $\partial\Omega$ given in (\ref{bgraph}) and let
\[
T_{\gamma}= \partial_{\gamma} + \sum_{\beta<n} B_{\gamma\beta} (x_{\beta}\partial_n - x_n \partial_{\beta})
\]
for $\gamma<n$ and $T_n = \partial_n$. Since $u=\underline u=\phi$ on $\partial\Omega$, we obtain
\[
|T_{\gamma}(u-\phi)| \leq C|x|^2 \ \text{ on } \partial\Omega_{\delta}, \text{ for } 1\leq \gamma\leq n-1.
\]
and, by differentiate equation (\ref{equ-G-form}),
\begin{eqnarray}\label{Talpha}
\left| G^{ij} \left(T_{\gamma}(u-\phi)\right)_{ij}\right| \leq C(1+\sum G^{ii})+ \sum_{\ell=0}^{k-2} \left|T_{\gamma}(\alpha_\ell) \right|\, {\sigma_\ell\over \sigma_{k-1}} \ \text{ in } \Omega_{\delta},
\end{eqnarray}
for $1\leq \gamma\leq n-1$. The last term in the above inequality arises from the non-constant coefficient functions. We can estimate it as following
\begin{eqnarray}
\sum_{\ell=0}^{k-2} \left|T_{\gamma}(\alpha_\ell) \right|\, {\sigma_\ell\over \sigma_{k-1}}\leq C \sum_{\ell=0}^{k-2} {\sigma_\ell\over \sigma_{k-1}} \leq C(1+ \sum_i G^{ii} |\lambda_i|).
\end{eqnarray}
The second inequality follows is derived from the non-degeneracy of the equation, more precisely
\begin{eqnarray}
\sum_i G^{ii} |\lambda_i| &\geq & \sum_i \left\{\left({\sigma_k\over \sigma_{k-1}}\right)^{ii} \, \lambda_i - \sum_{\ell=0}^{k-2} \alpha_\ell \, \left({\sigma_\ell \over \sigma_{k-1} }\right)^{ii} \, \lambda_i \right\}
\nonumber\\
&=&
{\sigma_k\over \sigma_{k-1}} + \sum_{\ell=0}^{k-2} \alpha_\ell \, {\sigma_\ell\over \sigma_{k-1} } + \sum_{\ell=0}^{k-2} \alpha_\ell\, (k-\ell) \, {\sigma_\ell\over \sigma_{k-1} }
\nonumber\\
&\geq &
- \alpha(x) + \delta \, \sum_{\ell=0}^{k-2} {\sigma_\ell\over \sigma_{k-1}}
\end{eqnarray}
where $\delta$ is a positive constant such that $\alpha_\ell\, (k-\ell) \geq \delta$ for $0\leq \ell \leq k-2$. Here we make use the assumption that $\alpha_\ell(x) >0$.

\smallskip

Putting the estimate into (\ref{Talpha}), we obtain
\begin{eqnarray}\label{Tgamma}
\left| G^{ij} \left(T_{\gamma}(u-\phi)\right)_{ij}\right| \leq C\left(1+\sum G^{ii} + \sum G^{ii}|\lambda_i|\right) \ \ \ \text{ in } \Omega_{\delta},
\end{eqnarray}
We remark that the term $\sum G^{ii}|\lambda_i|$ is new trouble term in comparison with the standard Hessian equation or quotient equation. Therefore, we need to construct a barrier function to overcome this term. Here, we use a function introduced by Guan \cite{GuanB1}
\begin{eqnarray}\label{Psi}
\Psi = A_1 v + A_2 |x|^2 - A_3 \sum_{\beta<n} \left|T_{\beta}(u- \phi)\right|^2
\end{eqnarray}
with $v$ giving in (\ref{barrierv}). Then, by using the Lemma \ref{keylemma} and taking $A_1\gg A_2 \gg A_3\gg 1$, we have
\begin{equation}
\begin{cases}
G^{ij} \left(\Psi \pm T_{\gamma}(u-\phi)\right)_{ij} \leq 0 & \text{ in } \Omega_{\delta}\\
\Psi\pm T_{\gamma}(u-\phi) \geq 0 & \text{ on } \partial \Omega_{\delta}.
\end{cases}
\end{equation}
The calculation follows the same as (4.9) in \cite{GuanB1}. The key observation is that the last term in $\Psi$ contributes a good term of the form $\sum G^{ii} |\lambda_i|$ which is used to overcome the new trouble term.
\noindent
Therefore, it follows from the maximum principle
\[
\Psi\geq |T_{\gamma}(u-\phi)|\ \ \ \textit{\rm in } \Omega_{\delta}.
\]Consequently,
\begin{eqnarray}\label{normaltangent}
|u_{\gamma n}(0) | \leq \Psi_{n}(0) + |\underline u_{\gamma n}(0)| \leq C, \ \text{ for } 1\leq \gamma<n.
\end{eqnarray}

\

\par
It remains to prove
\begin{eqnarray}\label{dnormal}
|u_{nn}(0)|\leq C.
\end{eqnarray}
Indeed, we only need to show the uniform upper bound $u_{nn}(0) <C$ since $\Gamma_{k-1} \subset \Gamma_1$ implies $\sum_i u_{ii}(0) \geq 0$ and then the lower bound for $u_{nn}(0)$ follows from the estimates (\ref{dtangent1}) and (\ref{normaltangent}). Again, we follow the main idea in \cite{GuanB1}, which was originally due to Trudinger \cite{Trudinger95}. To obtain the upper bound, we show that there are uniform constants $c_0, R_0$ such that, for all $R>R_0$, $\left(\lambda'(u_{\alpha\beta}(0)), R\right) \in \Gamma_{k-1}$ and
\begin{eqnarray}\label{unn}
G\left( \lambda'(u_{\gamma\beta}(0)), R \right) \geq -\alpha(0) + c_0.
\end{eqnarray}
Here, $\lambda'(u_{\gamma\beta}) = (\lambda'_1, \cdots, \lambda'_{n-1})$ denotes the eigenvalues of the $(n-1)\times (n-1)$ matrix $(u_{\gamma\beta})_{1\leq \gamma, \beta \leq n-1}$. Suppose that we have found such $c_0$ and $R_0$, by Lemma 1.2 in \cite{CNSIII}, it follows from estimate (\ref{dtangent1}) and (\ref{normaltangent}) that we can find $R_1\geq R_0$ such that, if $u_{nn}(0)> R_1$, then
\begin{eqnarray}
G\left(\lambda(u_{ij}(0))\right) \geq G\left( \lambda'(u_{\alpha\beta}(0)), u_{nn}(0) \right) - \frac{c_0}{2}.
\end{eqnarray}
where $\lambda(u_{ij}) = (\lambda_1, \cdots, \lambda_{n})$ denotes the eigenvalues of the $n\times n$ matrix $(u_{ij})_{1\leq i, j \leq n}$. However, the above two inequalities lead
\begin{eqnarray*}
 G\left( \lambda(u_{ij}(0))\right) \geq -\alpha(0) +\frac{c_0}{2}.
\end{eqnarray*}
which contradicts to the equation $G(\lambda(D^2u))(0)) = - \alpha(0)$.
Thus, we obtain the desired bound $u_{nn}(0)\leq R_1$.

\par
For a symmetric $(n-1)\times (n-1)$ matrix $(A_{\gamma\beta})_{1\leq \gamma,\beta\leq n-1}$ with $\left(\lambda'(A_{\gamma\beta}), R\right) \in \Gamma_{k-1}$ when $R>0$ is sufficiently large, we define
\begin{eqnarray*}
F_R(A_{\gamma\beta}) = G\left(\lambda'(A_{\gamma\beta}), R\right), \ \ \ \tilde{F}(A_{\gamma\beta}) =\lim_{R\to \infty} F_R(A_{\gamma\beta})
\end{eqnarray*}
and
\begin{eqnarray*}
m_{R} = \min_{x\in \partial\Omega}\left( F_R (A_{\gamma\beta})  +\alpha(x)\right), \ \ \ \tilde{m} = \lim_{R\to \infty} m_R.
\end{eqnarray*}
We want to show that
\begin{eqnarray}\label{mR}
\tilde{m} \geq c_0
\end{eqnarray}
for some uniform constant $c_0>0$. We assume $\tilde{m}< \infty$ for otherwise we are done. Suppose that $\tilde{m}$ is achieved at $x_0\in \partial\Omega$. Choose a local orthonormal frame around $x_0$ as before.

\noindent
First, we note that $\tilde{F}$ is finite and concave since $G$ is concave and continuous. Moreover, for any symmetric matrix $(A_{\gamma\beta})$ with $\left(\lambda'(A_{\gamma\beta}), R\right) \in \Gamma_{k-1}$, we have
\begin{eqnarray}
\tilde{F}(A_{\gamma\beta}) = \lim_{R\to \infty} G(\lambda'(A_{\gamma\beta}), R)={\sigma_{k-1}\over \sigma_{k-2}} (\lambda'(A_{\gamma\beta}))-\sum_{\ell=1}^{k-2} \alpha_\ell\, {\sigma_{\ell-1}\over \sigma_{k-2}}(\lambda'(A_{\gamma\beta}))
\end{eqnarray}
 Denote $\tilde{F}_0^{\gamma\beta}= {\partial \tilde{F}\over \partial A_{\gamma\beta}}(u_{\gamma\beta}(x_0))$. Using the concavity, we can compute
\begin{eqnarray}\label{tildeF}
&&\tilde{F}_0^{\gamma\beta} (u_{\gamma\beta}(x) - u_{\gamma\beta}(x_0))\\\nonumber
&=&
\left( {\sigma_{k-1}\over \sigma_{k-2}}\right)^{\gamma\beta}\bigg|_{x_0} \, (u_{\gamma\beta}(x) - u_{\gamma\beta}(x_0)) - \sum_{\ell=1}^{k-2}\alpha_\ell(x_0) \,\left({\sigma_{\ell-1}\over \sigma_{k-2}}\right)^{\gamma\beta}\bigg|_{x_0}\, (u_{\gamma\beta}(x) - u_{\gamma\beta}(x_0))
\\\nonumber
&\geq&
{\sigma_{k-1}\over \sigma_{k-2}}(u_{\gamma\beta}(x)) -{\sigma_{k-1}\over \sigma_{k-2}}(u_{\gamma\beta}(x_0))
- \sum_{\ell=1}^{k-2} \alpha_\ell(x_0)\, \left({\sigma_{\ell-1}\over \sigma_{k-2}}(u_{\gamma\beta}(x))-{\sigma_{\ell-1}\over \sigma_{k-2}}(u_{\gamma\beta}(x_0)) \right)
\\\nonumber
&=&
\tilde{F}(u_{\gamma\beta}(x)) - \tilde{F}(u_{\gamma\beta}(x_0)) + \sum_{\ell=0}^{k-2} \left(\alpha(x)- \alpha(x_0)\right) \,{\sigma_{\ell-1}\over \sigma_{k-2}}(u_{\gamma\beta}(x)).
\end{eqnarray}
In particular, this implies, on $\partial\Omega$,
\begin{eqnarray}\label{minimun1}
&&\tilde{F}_0^{\gamma\beta} u_{\gamma\beta}(x) +\alpha(x) - \tilde{F}_0^{\gamma\beta} u_{\gamma\beta}(x_0) -\alpha(x_0)
\\\nonumber
&\geq&
\tilde{F}(u_{\gamma\beta}(x)) + \alpha(x) - \tilde{m} -\sum_{\ell=0}^{k-2} \left(\alpha(x)- \alpha(x_0)\right) \,{\sigma_{\ell-1}\over \sigma_{k-2}}(u_{\gamma\beta}(x))
\geq -K|x-x_0|
\end{eqnarray}
for some constant $K$ depending on $\|\alpha_\ell\|_{C^1}$ and $\|u_{\gamma\beta}\|_{L^\infty}$ with $1\leq \gamma, \beta <n$.
\noindent
 Recall (\ref{dtangent}) on $\partial\Omega$, we have
\begin{eqnarray*}
u_{\gamma\beta}(x) = \underline u_{\gamma\beta}(x) - (u-\underline u)_n(x) B_{\gamma\beta}(x).
\end{eqnarray*}
It follows that
\begin{eqnarray*}
&&(u-\underline u)_n(x_0) \sum_{1\leq \gamma,\beta\leq n-1}B_{\gamma\beta}(x_0)\,\tilde{F}_0^{\gamma\beta}= \tilde{F}_0^{\gamma\beta} \left(\underline u_{\gamma\beta}(x_0) - u_{\gamma\beta}(x_0)\right) \\
&\geq & \tilde{F} (\underline u_{\gamma\beta}(x_0)) - \tilde{F}( u_{\gamma\beta}(x_0))=\tilde{F} (\underline u_{\gamma\beta}(x_0)) +\alpha(x_0) - \tilde{m}.
\end{eqnarray*}
We also note that, by the definition of subsolution $\underline u$ in (\ref{subsolution1}),
\begin{eqnarray*}
F_R (\underline u_{\gamma\beta}(x)) +\alpha(x) \geq F_R (\underline u_{\gamma\beta}(x)) - G\left( \lambda(\underline u_{ij}(x))\right) =  G\left(\lambda'(\underline u_{\gamma\beta}(x)), R\right)- G\left( \lambda(\underline u_{ij}(x))\right)
\end{eqnarray*}
which is strictly positive for $R>0$ large enough, by Lemma 2.1 in \cite{CNSIII} and the ellipticity. It follows that
\begin{eqnarray*}
(u-\underline u)_n(x_0) \sum_{1\leq \gamma,\beta\leq n-1}B_{\gamma\beta}(x_0)\tilde{F}_0^{\gamma\beta}\geq \tilde{c} - \tilde{m}
\end{eqnarray*}
where $\tilde{c} = \liminf_{R\to \infty} c_R$ with
$c_R = \min_{x\in\partial\Omega} \left( G\left(\lambda'(\underline u_{\gamma\beta}(x)), R\right)- G\left( \lambda(\underline u_{ij}(x))\right) \right)>0$.
Therefore, if
\[
(u-\underline u)_n(x_0) \sum_{1\leq \gamma,\beta\leq n-1}B_{\gamma\beta}(x_0)\tilde{F}_0^{\gamma\beta}\leq \frac{\tilde{c}}{2},
\] then $\tilde{m} \geq \frac{\tilde c}{2}$ and we are done. Suppose now that
\[
(u-\underline u)_n(x_0) \sum_{1\leq \gamma,\beta\leq n-1}B_{\gamma\beta}(x_0)\, \tilde F_0^{\gamma\beta} >\frac{\tilde c}{2}
\]
and denote $\eta(x)=\sum_{1\leq \gamma,\beta\leq n-1}B_{\gamma\beta}(x)\, \tilde F_0^{\gamma\beta}$. Note that
\begin{eqnarray}
\eta(x_0) \geq \frac{\tilde c}{2(u-\underline u)_n(x_0)}\geq 2 \epsilon_1\, \tilde c
\end{eqnarray}
for some uniform constant $\epsilon_1$ independent of $R$. We may assume that $\eta\geq \epsilon_1\, \tilde c$ on $\bar\Omega_{\delta}$ by requiring $\delta$ small enough. In $\Omega_{\delta}= \Omega\cap B_{\delta}(x_0)$, consider
\[
\Phi(x) = - (u- \phi)_n(x)+ \frac{1}{\eta(x)}  \tilde F_0^{\gamma\beta} \left( \phi_{\gamma\beta}(x)- u_{\gamma\beta}(x_0)\right)+\frac{1}{\eta(x)}(\alpha(x)- \alpha(x_0))+ {K\over \eta(x)} |x-x_0|
\]
where $K$ is the constant from (\ref{minimun1}). We can check that $\Phi(x)\geq 0$ for $x\in\partial\Omega$ near $x_0$ by using inequality (\ref{tildeF}).

It also follows that $\Phi(x_0) =0$.
Moreover, by (\ref{Tgamma}), we can compute
\begin{eqnarray*}
 G^{ij}\Phi_{ij} \leq C(1+ \sum_i G^{ii}+ \sum_i G^{ii}|\lambda_i|).
\end{eqnarray*}
Therefore, by using the key lemma and the barrier function $\Psi$ given in (\ref{Psi}), we can obtain
\begin{equation}
\begin{cases}
G^{ij} \left(\Psi+\Phi\right)_{ij} \leq 0 & \text{ in } \Omega_{\delta}\\
\Psi+\Phi \geq 0 & \text{ on } \partial \Omega_{\delta}.
\end{cases}
\end{equation}
Applying the maximum principle, we have $\Psi+\Phi \geq 0$ in $\Omega_{\delta}$ and then
\[
\Phi_n (x_0) \geq - \Psi_n(x_0) \geq -C
\]
which gives $u_{nn}(x_0) \leq C$.
\medskip
\par
As the final step, we need to show that this uniform upper bound holds at any point $x\in \partial\Omega$. From the discussion in (\ref{unn})-(\ref{mR}), it suffices to show that, if $R>0$ large enough,
\[
m_R = F_R\left(\lambda'(u_{\gamma\beta}(x_0)), R\right) + \alpha(x_0) \geq c_0.
\]
First, we note that the estimate $u_{nn}(x_0)<C$ together with (\ref{dtangent1}) and (\ref{normaltangent}) imply that all eigenvalues of $(u_{ij}(x_0))_{1\leq i, j\leq n}$ have {\it a priori} upper bound, which tells that eigenvalues of $(u_{ij}(x_0))_{1\leq i, j\leq n}$ are contained in $\Gamma_{k-1}\cap B_C(0)\subset \mathbb R^n$. On the other hand, we claim that the eigenvalues can not touch $\partial\Gamma_{k-1}$. If this is true, then, $\lambda(u_{ij}(x_0))$ is contained in a compact subset of $\Gamma_{k-1}$. Therefore, if $R>0$ large enough,
\[
m_R =  F_R\left(\lambda'(u_{\gamma\beta}(x_0)), R\right) + \alpha(x_0)>0.
\]
So, we only need to show that $\lambda(u_{ij}(x_0))$ can not touch $\partial\Gamma_{k-1}$. Indeed, this is the direct consequence of our non-degeneracy assumption. Recall our equation
\[
 G(\lambda(D^2 u)) = \frac{\sigma_k}{\sigma_{k-1} }- \sum_{\ell=0}^{k-2} \alpha_\ell(x) \frac{\sigma_\ell}{\sigma_{k-1}}=-\alpha(x).
\]
Note that, for $\lambda\in \Gamma_{k-1}$, we still have $\sigma_k(\lambda)\sigma_{k-2}(\lambda) \leq c(n, k) \sigma_{k-1}^2(\lambda)$. This gives
\begin{eqnarray*}
\frac{\sigma_k(\lambda)}{\sigma_{k-1}(\lambda)} \leq c(n, k) \frac{\sigma_{k-1}(\lambda)}{\sigma_{k-2}(\lambda)} \leq \tilde c(n, k) \frac{\sigma_{k-1}(\lambda)}{\sigma_{k-1}^{\frac{k-2}{k-1}}(\lambda)} =  \tilde c(n, k)\sigma_{k-1}^{\frac{1}{k-1}}(\lambda).
\end{eqnarray*}
Then,
\[
\frac{\sigma_k(\lambda)}{\sigma_{k-1}(\lambda)} \leq 0, \ \text{ as } \lambda \rightarrow \partial\Gamma_{k-1}.
\]
By the non-degeneracy assumption ($\alpha_{\ell}(x_0)>0$), if $\lambda(u_{ij}(x_0))\rightarrow \partial\Gamma_{k-1}$,
$G\left(\lambda(u_{ij}(x_0))\right) \to - \infty$,
since ${\sigma_0 \over \sigma_{k-1}}={1 \over \sigma_{k-1}} \to + \infty$ as $\lambda \rightarrow \partial\Gamma_{k-1}$. This contradicts with the condition that $\alpha(x) \in C^{1, 1}(\bar\Omega)$.

 \qed


\section{Further Remarks}

\par
We can also generalize the estimates to equations of the following form
\begin{eqnarray}\label{p1}
\quad \sigma_k\left(\chi(x) + W_u(x)\right) + \alpha(x) \sigma_{k-1} \left(\chi(x) + W_u(x)\right) = \sum_{\ell=0}^{k-2} \alpha_\ell(x) \sigma_{\ell}\left(\chi(x) + W_u(x)\right)
\end{eqnarray}
where $\chi(x)$ is a symmetric $2-$tensor defined on $\mathbb S^n$.
\smallskip
\par
Note that we only need to modify the $C^2$ estimate by considering the test function $\tilde H= {\rm tr}\chi+ H$. Same as before, the commutator term $H\sum_{i=1}^n G^{ii}$ is the good term and we can use this term to dominate the term comes from $\chi$, $\sum_{i=1}^n \chi_{ii} G^{ii}$ since $\chi_{ii}$ is bounded from below. Then, all the rest are the same.
\medskip
\par
By making use of this generalization, we can deal with the following general equations which are closely related to the special Lagrangian equations in dimension 3.

 \begin{eqnarray*}
\sigma_3(W_u)+ a(x) \sigma_2(W_u) + b(x) \sigma_1(W_u) + c(x)=0.
\end{eqnarray*}
Let $W_u = \lambda =\tilde\lambda + \tilde c(x) I_n$ and reduce the equation to the form
\[
\sigma_3(\tilde \lambda) + \alpha(x) \sigma_2(\tilde\lambda) = \beta(x)\sigma_1(\tilde\lambda) + \gamma(x).
\]
To make this equation fit into the frame (\ref{p1}), we want to make $\beta(x)=0$. Note that we have no assumption on the sign of $\alpha(x)$. Indeed, for this case, we do not need $\gamma(x)\geq 0$. Since if $\gamma(x)\leq 0$, then we consider $\tilde\lambda' = - \tilde\lambda$ and the equation reduce to be
\[
\sigma_3(\tilde\lambda' ) + \alpha(x) \sigma_2(\tilde\lambda') = - \gamma(x),
\]
with $-\gamma(x)\geq 0$ and it is solvable for $\tilde\lambda'$.

\noindent
To make sure that we can find $\tilde c(x)$ such that $\beta(x) =0$, we need to assume
\[
b(x) \leq \frac{(n-1)a^2(x)}{2(n-2)}.
\]
From this simple example, we see that it is important to release the sign requirement for $\alpha(x)$ because one might lead to an overdetermined system for the original coefficients $a(x), b(x)$ and $c(x)$ if we still have restriction on $\alpha(x)$.

\


{\small
\begin{thebibliography}{99}
\bibitem{Alex} A.D. Alexandrov,  {\em Zur Theorie der gemischten Volumina von konvexen K\"orpern, II. Neue Ungleichungen zwischen den gemischten Volumina und ihre Anwendungen}, Mat. Sb. 2, 1205-1238 (1937)
\smallskip
\bibitem{Alex1} A.D. Alexandrov, {\em Zur Theorie der gemischten Volumina von konvexen K\"orpern, III. Die Erweiterung zweier Lehrs\"atze Minkowskis \"uber die konvexen Polyeder auf beliebige konvexe Fl\"achen (in Russian)}, Mat. Sb. 3, 27-46 (1938)
\smallskip
\bibitem{Berg} C. Berg, {\em Corps convexes etpotentials sph\'eriques. Det. Kgl. DanskeV idenskab}, Selskab, Math. -fys. Medd. 37, 1-64 (1969)
\smallskip
\bibitem{CNSI} L. Caffarelli, L. Nirenberg, J. Spruck, {\em The Dirichlet problem for nonlinear second-order elliptic equations, I: Monge-Amp\`ere equation}, Comm. Pure and Appl. Math. 37, (1984), 369-402.
\bibitem{CKNSII} L. Caffarelli, L. Nirenberg, J. Spruck, {\em The Dirichlet problem for nonlinear second-order elliptic equations, II: Complex Monge-Amp\'ere equation and uniformly elliptic equations}, Comm. Pure Appl. Math., {\bf 38} (1985), 209-252.
\smallskip
\bibitem{CNSIII} L. Caffarelli, L. Nirenberg, J. Spruck, {\em The Dirichlet problem for nonlinear second-order elliptic equations, III: Functions of the eigenvalues of the Hessian}, Acta Math. {\bf 155} 261-301.
\smallskip
\bibitem{CNSIV} L. Caffarelli, L. Nirenberg, J. Spruck, {\em Nonlinear second-order elliptic equations, IV: Star-shaped compact Weingarten hypersurfaces}, Current topics in PDE, 1-26, Kinokuniya, Tokyo, 1986.
\smallskip
\bibitem{Cheng-Yau} S.-Y. Cheng, S.-T.Yau, {\em On the regularity of the solution of the n dimensional Minkowski problem}, Commun. Pure Appl. Math. 29, 495-516 (1976).
\smallskip
\bibitem{CJY} T.C. Collins, A. Jacob, S.-T. Yau, {\em $(1,1)$ forms with specified Lagrangian phase: A priori estimates and algebraic obstructions}, arXiv:1508.01934.
\smallskip
\bibitem{CS} T. Collins and G. Sz\'ekelyhidi, {\em Convergence of the J-flow on toric manifolds}, J. Differential Geom. 107 (2017) no. 1, 47-81
\smallskip
\bibitem{CY} T.C. Collins and S.-T. Yau, {\em Moment maps, nonlinear PDE, and stability in mirror symmetry}, arXiv:1811.04824.
\smallskip
\bibitem{DPZ} S. Dinew, S. Plis and X.W. Zhang, {\em Regularity of Degenerate Hessian Equations}, arXiv:1805.05761.

\smallskip
\bibitem{Dong} H.J. Dong, {\em Hessian equations with elementary symmetric functions}, Comm. Partial Differential Equations {\bf 31} (2006), 1005-1025.
\smallskip
\bibitem{Firey1} W.J. Firey, {\em The determination of convex bodies from their mean radius of curvature functions}, Mathematik 14, 1-14 (1967)
\smallskip
\bibitem{Firey2} W.J. Firey, {\em Christoffel problems for general convex bodies}, Mathematik 15, 7-21 (1968)
\smallskip
\bibitem{FY1} J.X. Fu and S.-T. Yau, {\it The theory of superstring with flux on non-K\"ahler manifolds and the complex Monge-Amp\`ere equation}, J. Differential Geom., Vol 78, Number 3 (2008), 369-428.
\smallskip
\bibitem{FY2} J.X. Fu and S.T. Yau, {\it A Monge-Amp\`ere type equation motivated by string theory}, Comm. Anal. Geom. 15 (2007), no. 1, 29-76.
\smallskip

\bibitem{Garding} L. G\aa rding, {\em An inequality for hyperbolic polynomials}, J. Math. Mech. {\bf 8} (1959), 957-965.
\smallskip
\bibitem{GuanB1} B. Guan, {\em The Dirichlet problem for fully nonlinear elliptic equations on Riemannian manifolds}, arXiv:1403.2133.
\smallskip
\bibitem{GuanB2} B. Guan, {\em Second order estimates and regularity for fully nonlinear elliptic equations on Riemannian manifolds}, Duke Math. J. 163 (2014), 1491-1524.
\smallskip
\bibitem{Guan-Guan}B. Guan and P. Guan, {\em Convex Hypersurfaces of Prescribed Curvatures}.  Annual of Mathematics, Vol 256, No. 2 (2002), 655-673.
\smallskip
\bibitem{Guan-Ma} P. Guan and X.N. Ma, {\em Christoffel-Minkowski Problem I: Convexity of Solutions of a Hessian Equations}, Inventions Math., 151, (2003), 553-577.
\smallskip
\bibitem{Guan} P. Guan, {\em Topics in Geometric Fully Nonlinear Equations}, 2004.
\smallskip
\bibitem{Guan-Li-Li} P. Guan, Y. Li and J. Li, {\em Hypersurfaces of prescribed curvature measure}. Duke Math. J., Vol. 161, No. 10, (2012), 1927-1942.
\smallskip
\bibitem{Guan-Ma-Zhou} P. Guan, X. Ma and F. Zhou, {\em The Christoffel-Minkowski problem III: existence and convexity of admissible solutions}, Communications on Pure and Applied Mathematics, V.59, (2006) 1352-1376.
\smallskip
\bibitem{GTW} P. Guan, N. Trudinger, X.J. Wang, {\em On the Dirichlet problem for degenerate Monge-Amp\`ere equations}, Acta Math. {\bf 182}(1), 87-104.
\smallskip
\bibitem{HL} R. Harvey, H.B. Lawson, {\em Calibrated geometries}. Acta. Math., 148 (1982), 47-157.
\smallskip
\bibitem{HS} G. Huisken and C. Sinestrari, {\em Convexity estimates for mean curvature flow and singularities
of mean convex surfaces}, Acta Math. 183 (1999), no. 1, 45-70.
\smallskip
\bibitem{Ivochkina80} N. Ivochkina, {\em  A priori estimate of $|u|_{C^2(\bar\Omega)}$ of convex solutions of the Dirichlet problem for the Monge-Amp\`ere equation}. Zap. Nauchn. Sem. Leningrad. otdel. Mat. Inst. Steklov. 96, (1980) 69-79 (Russian); English transl. J. Soviet Math. (1983), 21, 689-697.
\smallskip
\bibitem{ITW} N. Ivochkina, N.S. Trudinger, X.J. Wang, {\em The Dirichlet problem for degenerate Hessian equations}, Comm. Partial Differential Equations {\bf 29} (1-2), 219-235.
\smallskip
\bibitem{JY} A. Jacob, S.-T. Yau, {\em A special Lagrangian type equation for holomorphic line bundles}. arXiv:1411.7457.


\bibitem{Krylov95} N. V. Krylov, {\em On the general notion of fully nonlinear second order elliptic equation}. Trans. Amer. Math. Soc. {\bf 347} (3), 857-895.
\smallskip
\bibitem{LYZ} C. Leung, S.-T. Yau, and E. Zaslow, {\em From special Lagrangian to Hermitian-Yang-Mills via Fourier-Mukai transform}, Winter School on Mirror Symmetry, Vector Bundles and Lagrangian Submanifolds, (1999), 209-225, AMS. IP Stud. Adv. Math., 23, Amer, Math. Soc., Providence, RI, 2001.
\smallskip
\bibitem{Lewy} H. Lewy, {\em On differential geometry in the large}, Trans. Am. Math. Soc. 43, 258-270 (1938)
\smallskip
\bibitem{LRW} C.H. Li, C.Y. Ren and Z.Z. Wang, {\em The curvature estimates for convex solutions of some fully nonlinear Hessian type equations}, arXiv:1705.09891
\smallskip
\bibitem{Li89} Y.Y. Li, {\em Degree theory for second order nonlinear elliptic operators and its applications}. Comm. Partial Differential Equations {\bf 14} (1989), no. 11, 1541-1578.
\smallskip
\bibitem{Mink} H. Minkowski, {\em Allgemeine Lehrs\"atze \"uber die konvexen Polyeder}. Nachr. Ges. Wiss. G\"ottingen 198-219 (1897)
\smallskip
\bibitem{Nirenberg0} L. Nirenberg, {\em The Weyl and Minkowski problems in differential geometry in thelarge}, Commun. Pure Appl. Math. 6, 337-394 (1953)
\smallskip
\bibitem{Nirenberg} L. Nirenberg, {\em Topics in Nonlinear Functional Analysis}, Courant Lecture notes in Mathematics.
\smallskip


\bibitem{PPZ1}
D.H. Phong, S. Picard, and X.W. Zhang, {\it On estimates for the Fu-Yau generalization of a Strominger system}, J. Reine Angew. Math. (Crelle's Journal), Vol. 2019, Issue 751 (2019), 243-274
\smallskip
\bibitem{PPZ} D.H. Phong, S. Picard, X.W. Zhang, {\em The Fu-Yau equation with negative slope parameter}, Invent. Math., Vol. 209, No. 2 (2017), 541-576.
\smallskip
\bibitem{PPZ2}
D.H. Phong, S. Picard, and X.W. Zhang, {\it Fu-Yau Hessian Equations}, arXiv:1801.09842, to appear in J. Differential Geometry.


\smallskip
\bibitem{Pog} A.V. Pogorelov, {\em The Minkowski multi-dimensional problem}, Wiley, NewYork 1978.
\smallskip
\bibitem{Schneider} R. Schneider, {\em Convex bodies: The Brunn-Minkowski theory}, Cambridge University, (1993).
\smallskip
\bibitem{Trudinger95} N.S. Trudinger, {\em On the Dirichlet problems for Hessian equations}, Acta Math. {\bf 175} 151-164.
\smallskip
\bibitem{Trudinger97} N.S. Trudinger, {\em Weak solutions of Hessian equations}. Comm. Partial Differential Equations 22, (1997) 1251-1261.
\smallskip
\bibitem{Wang95} X.J. Wang, {\em Some counterexamples to the regularity of Monge-Amp\`er equations}. Proc. Amer. Math. Soc. 123, (1995) 841-845.
\smallskip
\bibitem{Wang09} X.J. Wang, {\em The $k$-Hessian equation}. Lectures Notes in Mathematics, Vol.1977, (2009) 177-252.
\smallskip
\bibitem{Yau1} S.-T. Yau, {\em On the Ricci curvature of a compact K\"ahler manifold and the complex Monge-Amp\`ere equation}, Comm. Pure Appl. Math. 31(1978), 339-411.

\end{thebibliography}
}
 \end{document}